\documentclass{amsart}

\usepackage[all]{xy}

\usepackage{hyperref}

\usepackage{cleveref}

\usepackage{amssymb}

\usepackage{todonotes}

\usepackage{relsize}
\usepackage[bbgreekl]{mathbbol}
\usepackage{amsfonts}

\DeclareSymbolFontAlphabet{\mathbb}{AMSb}
\DeclareSymbolFontAlphabet{\mathbbl}{bbold}
\newcommand{\prism}{{\mathlarger{\mathbbl{\Delta}}}}

\theoremstyle{plain}

\newtheorem{theorem}{Theorem}[section]
\newtheorem{proposition}[theorem]{Proposition}
\newtheorem{lemma}[theorem]{Lemma}
\newtheorem{corollary}[theorem]{Corollary}

\theoremstyle{definition}
\newtheorem{definition}[theorem]{Definition}

\newtheorem{remark}[theorem]{Remark}

 \newcommand{\crys}{\mathrm{crys}}
\newcommand{\cycl}{\mathrm{cycl}} 
 \newcommand{\F}{\mathbb{F}}
 
 \newcommand{\GL}{\mathrm{GL}}
 \newcommand{\HC}{\mathrm{HC}} \newcommand{\HH}{\mathrm{HH}}
 
 \renewcommand{\inf}{\mathrm{inf}}

\newcommand{\Q}{\mathbb{Q}} 
 
 \newcommand{\Spec}{\mathrm{Spec}}
\newcommand{\TC}{\mathrm{TC}} \newcommand{\THH}{\mathrm{THH}}
\newcommand{\TP}{\mathrm{TP}} \newcommand{\tilxi}{\tilde{\xi}}
 \newcommand{\Z}{\mathbb{Z}}

\newcommand{\blue}[1]{{\color{black} #1}}

\begin{document}

\title{The $p$-completed cyclotomic trace in degree $2$}

\author{Johannes Ansch\"{u}tz}
\address{Institut de Math\'ematiques de Jussieu, 4, place Jussieu, 
75252 Paris cedex 05, France}
\email{ja@math.uni-bonn.de} 
\thanks{During this project, J.A.\ was partially supported by the ERC 742608, GeoLocLang.}

\author{Arthur-C\'esar Le Bras}
\address{Universit\'e Sorbonne Paris Nord, LAGA, C.N.R.S., UMR 7539, 93430 Villetaneuse,
France}
 \email{lebras@math.univ-paris13.fr}

\begin{abstract}
We prove that for a quasi-regular semiperfectoid  $\Z_p^{\rm cycl}$-algebra $R$ (in the sense of Bhatt-Morrow-Scholze), the cyclotomic trace map from the $p$-completed $K$-theory spectrum $K(R;\Z_p)$ of $R$ to the
topological cyclic homology $\TC(R;\Z_p)$ of $R$ identifies on $\pi_2$ with a $q$-deformation of the logarithm.
\end{abstract}

\maketitle
\tableofcontents

\section{Introduction}
\label{sec:introduction}

Fix a prime $p$. The aim of this paper is to concretely identify in
degree $2$, for a certain class of $p$-complete rings $R$, the
$p$-completed cyclotomic trace
$$
\mathrm{ctr}\colon K(R;\Z_p)\to \TC(R;\Z_p)
$$
from the $p$-completed $K$-theory spectrum $K(R;\Z_p)$ of $R$ to the
topological cyclic homology $\TC(R;\Z_p)$ of $R$. Our main result is that on $\pi_2$ the $p$-completed cyclotomic trace is given by a $q$-logarithm
$$
\log_q(x):=\sum\limits_{n=1}^\infty
(-1)^{n-1}q^{-n(n-1)/2}\frac{(x-1)(x-q)\cdots (x-q^{n-1})}{[n]_q},
$$
which is a $q$-deformation of the usual logarithm (where $q$ is a parameter which will be defined later). Before stating a precise version of the theorem, let us try to put it in context and to explain what the involved objects are.

\subsection{$K$-theory and topological cyclic homology}
We start with $K$-theory. For any commutative ring $A$, Quillen defined in \cite{quillen_higher_algebraic_k_theory_I} the algebraic $K$-theory space $K(A)$
of $A$ as a generalization of the Grothendieck group $K_0(A)$ of vector bundles on the scheme $\mathrm{Spec}(A)$. The (connective) $K$-theory spectrum $K(A)$ of a ring $A$ is obtained by group completing\footnote{Cf.\ \cite{nikolaus_the_group_completion_theorem_via_localization_of_ring_spectra} for a discussion of homotopy-theoretic group completions and Quillen's $+$-construction.} the $\mathbb{E}_{\infty}$-monoid of vector bundles on $\Spec(A)$ whose addition is given by the direct sum. In other words, for the full $K$-theory one mimicks in a homotopy theoretic context the definition of $K_0(A)$ with the set of isomorphism classes of vector bundles replaced by the groupoid of vector bundles. Algebraic $K$-theory behaves like a cohomology theory but has the nice feature, compared to other cohomology theories, like \'etale cohomology, that it only depends on the category of vector bundles on the ring (rather than on the ring itself) and thus enjoys strong functoriality properties, which makes it a powerful invariant attached to $A$.

Unfortunately, the calculation of the homotopy groups
$$K_i(A):=\pi_i(K(A)), ~~ i\geq 1,$$ is in general rather untractable. There is for example a natural embedding
$$
A^\times\to \pi_1(K(A)),
$$
which is an isomorphism if $A$ is local, but the higher $K$-groups are much more mysterious. One essential difficulty comes from the fact that $K$-theory, although it is a Zariski (and even Nisnevich) sheaf of spaces (cf.\ \cite{thomason_higher_algebraic_k_theory_of_schemes_and_of_derived_categories}), does not satisfy \'etale descent. One could remedy this by \'etale sheafification, but one would lose the good properties of $K$-theory. This lead people to look for good \textit{approximations} of $K$-theory, at least after profinite completion : by this, we mean invariants, still depending only on the category of vector bundles on the underlying ring, satisfying \'etale descent - and therefore, easier to compute - and close enough to (completed) $K$-theory, at least in some range.

The work of Thomason, \cite{thomason_algebraic_k_theory_and_etale_cohomology}, provides a good illustration of this principle. Thomason shows that the \textit{$K(1)$-localization} of $K$-theory, with respect to a prime $\ell$ \textit{invertible in} $A$, satisfies \'etale descent\footnote{In fact, it even coincides with $\ell$-adic \'etale $K$-theory on connective covers.} and coincides with $\ell$-adically completed (for short: $\ell$-adic) $K$-theory in high degrees under some extra assumptions, later removed by Rosenschon-Ostaver, \cite{rosenschon_ostvaer_descent_for_k_theories}, buliding upon the work of Voevodsky-Rost. When the prime $p$ is not invertible in $A$, the situation is much more subtle. For instance, a theorem of Gabber \cite{gabber_k_theory_of_henselian_local_rings_and_henselian_pairs} shows that $\ell$-adic $K$-theory is insensitive to replacing $A$ by $A/I$ if $(A,I)$ forms an henselian pair ; in particular, the computation of $\ell$-adic K-theory of henselian rings (which form a basis of the Nisnevich topology) is reduced to the computation of the $\ell$-adic $K$-theory of fields. This is not true anymore for $p$-adic $K$-theory. Nevertheless, the recent work of Clausen-Mathew-Morrow, \cite{clausen_mathew_morrow_k_theory_and_topological_cyclic_homology_of_henselian_pairs}, expresses this failure in terms of another non-commutative invariant attached to $A$, the \textit{topological cyclic homology} of $A$, whose definition will be recalled below. Topological cyclic homology is related to $K$-theory via the \textit{cyclotomic trace}, cf.\ \cite[Section 10.3]{blumberg_gepner_tabuada_a_universal_characterization_of_higher_algebraic_k_theory}, \cite[Section 5]{boekstedt_hsiang_madsen_the_cyclotomic_trace_and_algebraic_k_theory_of_spaces},
$$
\textrm{ctr}\colon K(A)\to \TC(A).
$$
Clausen, Mathew and Morrow prove, extending earlier work of Dundas, Goodwillie and McCarthy \cite{dundas_goodwillie_mccarthy_the_local_structure_of_algebraic_k_theory} in the nilpotent case\footnote{This is not a generalization though, since the result of Dundas-Goodwillie-McCarthy applies also to non-commutative rings and is not restricted to finite coefficients.}, that the cyclotomic trace induces, for any ideal $I\subseteq A$ such that the pair $(A,I)$ is henselian, an isomorphism $$K(A,I)/n\cong \TC(A,I)/n$$ from the relative $K$-theory
$$
K(A,I)/n:=\mathrm{fib}(K(A)/n\to K(A/I)/n)
$$
to the relative topological cyclic homology
$$
\TC(A,I)/n:=\mathrm{fib}(\TC(A)/n\to \TC(A/I)/n),
$$
for \textit{any} integer $n$. This has the consequence that $p$-completed $\TC$ provides a good approximation of $p$-adic $K$-theory, at least for rings henselian along $(p)$: namely, it satisfies \'etale descent (because topological cyclic homology does) and coincides with $p$-adic $K$-theory in high degrees. Under additional hypotheses, one can even get better results: for instance, Clausen, Mathew and Morrow prove, among other things, that the cyclotomic trace induces an
isomorphism $$K(R;\Z_p)\cong \tau_{\geq 0}\TC(R;\Z_p)$$ for all rings $R$ which are henselian
along $(p)$ and such that $R/p$ is semiperfect (i.e., such that Frobenius is surjective), cf.\ \cite[Corollary 6.9.]{clausen_mathew_morrow_k_theory_and_topological_cyclic_homology_of_henselian_pairs}.

Examples of such rings are the \textit{quasi-regular semiperfectoid rings} of \cite{bhatt_morrow_scholze_topological_hochschild_homology}. A ring $R$ is called
quasi-regular semiperfectoid, if $R$ is $p$-complete with
bounded $p^\infty$-torsion\footnote{This means that there exists $N\geq 0$ such that $R[p^{\infty}]=R[p^N]$. This technical condition is useful when dealing with derived completions.}, the $p$-completed cotangent complex
$\widehat{L_{R/\Z_p}}$ has $p$-complete Tor-amplitude in $[-1,0]$ and there
exists a surjective morphism $R^\prime\to R$ with $R^\prime$
(integral) perfectoid. This class of rings is interesting
as for $R$ quasi-regular semiperfectoid the topological cyclic homology
$\pi_\ast(\TC(R;\Z_p))$ can be computed in more concrete terms.

Let us recall the description of topological cyclic homology $\pi_\ast(\TC(R;\Z_p))$ from \cite{bhatt_morrow_scholze_topological_hochschild_homology}, which builds
heavily on the foundational work of Nikolaus and Scholze \cite{nikolaus_scholze_on_topological_cyclic_homology}. For this, we need to spell some definitions. From now on, all spectra will be assumed to be $p$-completed. One starts with the ($p$-completed)
topological Hochschild homology spectrum $\THH(R;\Z_p)$ of $R$, which is
equipped with a natural $\mathbb{T}=S^1$-action and a
$\mathbb{T}$-equivariant map, the cyclotomic Frobenius,
$$
\varphi_{\cycl}\colon \THH(R;\Z_p)\to \THH(R;\Z_p)^{tC_p}
$$
to the Tate fixed points of the cyclic group
$C_p\subseteq \mathbb{T}$. Then one takes the homotopy fixed points,
the \textit{negative topological cyclic homology},
$$
\TC^{-}(R;\Z_p):=\THH(R;\Z_p)^{h\mathbb{T}}
$$
and the Tate fixed points, the \textit{periodic topological cyclic homology},
$$
\TP(R;\Z_p):=\THH(R;\Z_p)^{t\mathbb{T}}.
$$
From the cyclotomic Frobenius on $\THH(R;\Z_p)$ one derives a map\footnote{Here one needs \cite[Lemma II.4.2.]{nikolaus_scholze_on_topological_cyclic_homology} which implies $\TP(R;\Z_p)\cong (\THH(R;\Z_p)^{tC_p})^{h\mathbb{T}}$.}
$$
\varphi^{h\mathbb{T}}_\cycl\colon \TC^{-}(R;\Z_p)\to \TP(R;\Z_p).
$$
Then the topological cyclic homology $\TC(R;\Z_p)$ of $R$ is defined via the fiber sequence
$$
\TC(R;\Z_p)\to
\TC^{-}(R;\Z_p)\xrightarrow{\mathrm{can}-\varphi_\cycl^{h\mathbb{T}}}\TP(R;\Z_p),
$$
where $\mathrm{can}\colon \TC^{-}(R;\Z_p)\to \TP(R;\Z_p)$ is the canonical map
from homotopy to Tate fixed points. The ring
$$
\widehat{\prism}_R:=\pi_0(\TC^{-}(R;\Z_p))\cong \pi_0(\TP(R;\Z_p)).
$$
is $p$-complete, $p$-torsion free\blue{\footnote{Indeed, any element killed by $p$ is killed by $\varphi$, cf. the proof of \cite[Lemma 2.28]{bhatt_scholze_prisms_and_prismatic_cohomology}, and thus lies in all the steps of the Nygaard filtration.}} and the cyclotomic Frobenius
$\varphi^{h\mathbb{T}}_\cycl$ induces a Frobenius lift $\varphi$ on
$\widehat{\prism}_R$ (cf.\ \cite[Theorem 11.10]{bhatt_scholze_prisms_and_prismatic_cohomology}).

\begin{remark}
  \label{sec:k-theory-topological-remark-topological-and-concrete-prism}
The prismatic perspective of \cite{bhatt_scholze_prisms_and_prismatic_cohomology} gives an alternative description of $\widehat{\prism}_R$ : it is the completion with respect to the Nygaard filtration of the (derived) prismatic cohomology $\prism_R$ of $R$. In particular, using the theory of $\delta$-rings, one can give, when $R$ is a $p$-complete with bounded $p^{\infty}$-torsion quotient of a perfectoid ring by a regular sequence, a construction of $\widehat{\prism}_R$ as the Nygaard completion of a concrete prismatic envelope (cf. \cite[Proposition 3.12]{bhatt_scholze_prisms_and_prismatic_cohomology}). 
\end{remark}

The choice of a morphism $R^\prime\to R$
with $R^\prime$ perfectoid yields a distinguished element $\tilxi$ (\blue{up to a unit}) of the
ring $\widehat{\prism}_R$. Using $\tilxi$ one defines the Nygaard filtration
$$
\mathcal{N}^{\geq i}\widehat{\prism}_R:=\varphi^{-1}((\tilxi^i))
$$
on $\widehat{\prism}_R$. The graded
rings $\pi_\ast(\TC^-(R;\Z_p))$ and $\pi_\ast(\TP(R;\Z_p))$ are then concentrated in
even degrees and
$$
\begin{matrix}
  \pi_{2i}(\TC^-(R;\Z_p))\cong \mathcal{N}^{\geq i}\widehat{\prism}_R\\
  \pi_{2i}(\TP(R;\Z_p))\cong \widehat{\prism}_R
\end{matrix}
$$
for $i\in \Z$ (cf.\ \cite[Theorem 11.10]{bhatt_scholze_prisms_and_prismatic_cohomology}).\footnote{These identifications depend on the choice of a suitable generator $v\in \pi_{-2}(\TC^-(R;\Z_p))$. If $R$ is an algebra over $\Z_p^\cycl$ we will clarify our choice in \Cref{sec:p-compl-cycl} carefully.} Moreover, on $\pi_{2i}$ the cyclotomic Frobenius
$$
\varphi_\cycl^{h\mathbb{T}}\colon \pi_{2i}(\TC^-(R;\Z_p))\to \pi_{2i}(\TP(R;\Z_p))
$$
identifies with the divided Frobenius $\frac{\varphi}{\tilxi^i}$. Thus from the definition
of $\TC(R;\Z_p)$ we obtain exact sequences
$$
0\to \pi_{2i}(\TC(R;\Z_p))\cong {\widehat{\prism}_R}^{\varphi=\tilxi^i}\to
\mathcal{N}^{\geq
  i}\widehat{\prism}_R\xrightarrow{1-\frac{\varphi}{\tilxi^i}}
\widehat{\prism}_R\to \pi_{2i-1}(\TC(R;\Z_p))\to 0.
$$
As mentioned in \Cref{sec:k-theory-topological-remark-topological-and-concrete-prism}, the ring $\widehat{\prism}_R$ tends to be computable. For example, if $R$ is
perfectoid, then $\widehat{\prism}_R\cong A_{\inf}(R)$ is Fontaine's
construction applied to $R$ and if $pR=0$, then $\widehat{\prism}_R$ is
the Nygaard completion of the universal PD-thickening $A_{\crys}(R)$ of
$R$. Thus, for quasi-regular semiperfectoid rings the target of the cyclotomic trace is rather explicit.

\subsection{Main results}
The results of \cite{clausen_mathew_morrow_k_theory_and_topological_cyclic_homology_of_henselian_pairs} (together with those of \cite{bhatt_morrow_scholze_topological_hochschild_homology}) therefore give a way of computing higher $p$-completed $K$-groups of quasi-regular semiperfectoid rings. But there is at least one degree (except $0$) where one can be more explicit, without using the cyclotomic trace map: namely, after
$p$-completion of $K(R)$ there is a canonical morphism
$$
T_p(R^\times)\to \pi_2(K(R;\Z_p))
$$
from the Tate module $T_p(R^\times)$ of the units of $R$, which is an isomorphism in many
cases. The results explained in the previous paragraph show that the cyclotomic trace identifies $\pi_2(K(R;\Z_p))$ with
\[ \pi_2(\TC(R;\Z_p))\cong
  \widehat{\prism}_R^{\varphi=\tilxi}.
 \]
What does the composite map
 $$
  T_p(R^\times)\to
  \pi_2(K(R;\Z_p))\xrightarrow{\mathrm{ctr}}\pi_2(\TC(R;\Z_p))\cong
  \widehat{\prism}_R^{\varphi=\tilxi}
$$
look like? The main result of this paper, which we now state, provides a concrete description of it. Let $R$ be a quasi-regular semiperfectoid ring which admits a compatible system of
morphisms $\Z[\zeta_{p^n}]\to R$ for $n\geq 0$. These morphisms give rise to the
elements
$$
\varepsilon=(1,\zeta_p,\ldots)\in
R^\flat=\varprojlim\limits_{x\mapsto x^p} R\ ,\ q:=[\varepsilon]_{\theta}\in \widehat{\prism}_R
$$
and
$$
\tilxi:=\frac{q^p-1}{q-1}.
$$
Here
$$
[-]_{\theta}\colon R^\flat\to \prism_R
$$
is the Teichm\"uller lift coming from the surjection $\theta \colon \prism_R \to R$ (see the proof of \Cref{lemma_p_power_compatible_elements_in_the_reduction}).

\begin{theorem}[cf.\ \Cref{sec:conclusion-theorem-identification-of-cyclotomic-trace}]
  \label{sec:introduction-concrete-cyclotomic-trace-introduction}

  The composition\footnote{Cf.\ \Cref{sec:p-compl-cycl} for a more precise description of the isomorphism $\pi_2(\TC(R;\Z_p))\cong \widehat{\prism}_R^{\varphi=\tilxi}$. We note that it depends on the choice of some compatible system $\varepsilon=(1,\zeta_p,\zeta_{p^2},\ldots)$ of primitive $p^n$-th roots of unity.}
  $$
  T_p(R^\times)\to
  \pi_2(K(R;\Z_p))\xrightarrow{\mathrm{ctr}}\pi_2(\TC(R;\Z_p))\cong
  \widehat{\prism}_R^{\varphi=\tilxi}
$$
is given by the $q$-logarithm
$$
x\mapsto \log_q([x]_{\theta}):=\sum\limits_{n=1}^{\infty}
(-1)^{n-1}q^{-n(n-1)/2}\frac{([x]_{\theta}-1)([x]_{\theta}-q)\ldots
  ([x]_{\theta}-q^{n-1})}{[n]_q}.
$$
\end{theorem}

Here we embed
$$
T_p(R^\times)\subseteq R^\flat,\ (r_0\in R^\times[p],r_1,\ldots)\mapsto (1,r_0,r_1,\ldots).
$$
By
$$
[n]_q:=\frac{q^n-1}{q-1}
$$
we denote the the $q$-analog of $n\in \Z$.

\begin{remark}
  \label{sec:main-results-remark-geisser-hesselholt}
A similar result can be found in \cite[Lemma 4.2.3.]{geisser_hesselholt_topological_cyclic_homology_of_schemes}, but only
before $p$-completion, on $\pi_1$ and in terms of $\mathrm{TR}_\ast$, which is not enough to deduce \Cref{sec:introduction-concrete-cyclotomic-trace-introduction} from their result.
\end{remark}

As a consequence of \cite{clausen_mathew_morrow_k_theory_and_topological_cyclic_homology_of_henselian_pairs} and \Cref{sec:introduction-concrete-cyclotomic-trace-introduction}, one gets the following result.

\begin{corollary}
\label{sec:introduction-corollary-bijectivity-qlog}
Let $R$ be a quasi-regular semiperfectoid $\Z_p^{\rm cycl}$-algebra. The map
\[ \log_q([-]_{\theta}) : T_p(R^{\times}) \to \widehat{\prism}_R^{\varphi=\tilxi} \]
is a bijection.
\end{corollary}
This corollary is used in \cite{anschuetz_le_bras_prismatic_dieudonne_theory}, which studies a prismatic version of Dieudonn\'e theory for $p$-divisible groups, and was our original motivation for proving \Cref{sec:introduction-concrete-cyclotomic-trace-introduction}.
\\

Here is a short description of the proof
of \Cref{sec:introduction-concrete-cyclotomic-trace-introduction}. By testing the universal case
$R=\Z^{\mathrm{cycl}}\langle x^{1/p^\infty}\rangle/(x-1)$ one is
reduced to the case where $(p,\tilxi)$ form a regular sequence on
$\widehat{\prism}_R$, i.e., the prism $(\widehat{\prism}_R,\tilxi)$ is
\textit{transversal} (cf.\ \Cref{sec:some-results-about-definition-transversal-prism}). In this situation, we prove that the reduction map
$$
\widehat{\prism}_R^{\varphi=\tilxi}\hookrightarrow \mathcal{N}^{\geq
  1}\widehat{\prism}_R/\mathcal{N}^{\geq 2}\widehat{\prism}_R
$$
is \textit{injective} (cf.\ \Cref{sec:transversal-prisms-corollary-injectivity-for-reduction-modulo-second-nygaard-filtration}). Thus it suffices to identify the composition
$$
T_p(R^\times)\xrightarrow{\mathrm{ctr}}
{\widehat{\prism}_R}^{\varphi=\tilxi}\to \mathcal{N}^{\geq
  1}\widehat{\prism}_R/\mathcal{N}^{\geq 2}\widehat{\prism}_R.
$$
Using the results \cite{bhatt_morrow_scholze_topological_hochschild_homology} the quotient
$\mathcal{N}^{\geq 1}\widehat{\prism}_R/\mathcal{N}^{\geq
  2}\widehat{\prism}_R$ identifies with the $p$-completed Hochschild
homology $\pi_2(\HH(R;\Z_p))$ (cf.\ \Cref{sec:prism-cohom-topol}) and
therefore the above composition identifies with the $p$-completed
Dennis trace. A straightforward computation then identifies the
$p$-completed Dennis trace (cf.\ \Cref{sec:p-completed-dennis}), which allows us to conclude. We
expect the results in \Cref{sec:p-completed-dennis} to be known, in some form, to
the experts, but we did not find the results anywhere in the literature.
\\

Let us end this introduction by a remark and a question. One could try to reverse the perspective from \Cref{sec:introduction-corollary-bijectivity-qlog} and try to recover a (very) special case of the result of Clausen-Mathew-Morrow (cf.\ \cite{clausen_mathew_morrow_k_theory_and_topological_cyclic_homology_of_henselian_pairs}) regarding the cyclotomic trace map using the concrete description furnished by \Cref{sec:introduction-concrete-cyclotomic-trace-introduction}. If $R$ is of characteristic $p$, we \blue{have} $q=1$ and then
the $q$-logarithm becomes the honest logarithm
$$
\log([-]_{\theta})\colon T_p(R^\times)\to A_\crys(R)^{\varphi=p}.
$$
In
\cite{scholze_weinstein_moduli_of_p_divisible_groups}, it is proven (using the exponential) that
the map $\log([-])$ is an isomorphism, when $R$ is the quotient
of a perfect ring modulo a regular sequence. If $R$ is the quotient of a perfectoid ring by a finite regular sequence and is $p$-torsion free, it is not difficult to deduce from Scholze-Weinstein's result that the map $$
  \log_q([-]_{\theta})\colon T_p(R^\times)\to \widehat{\prism}_R^{\varphi=\tilxi}
  $$
  is a bijection when $p$ is odd. Is there a way to prove it directly in general, for any $p$ and any quasi-regular semiperfectoid ring ?

\subsection{Plan of the paper}
In \Cref{sec:p-completed-dennis} we concretely identify the $p$-completed Dennis trace on the Tate module of units (cf.\ \Cref{proposition_second_description_of_dennis_trace}) in the form we need it. In \Cref{sec:some-results-prisms} we prove the crucial injectivity statement, namely \Cref{sec:transversal-prisms-corollary-injectivity-for-reduction-modulo-second-nygaard-filtration}, for transversal prisms. In \Cref{section-q-logarithm} we make sense of the $q$-logarithm. Finally, in \Cref{sec:p-compl-cycl} we prove our main result \Cref{sec:introduction-concrete-cyclotomic-trace-introduction} and its consequence, \Cref{sec:introduction-corollary-bijectivity-qlog}.

\subsection{Acknowledgements}
\label{sec:acknowledgments}
  The authors thank Peter Scholze for answering several questions and his suggestion to think about \Cref{lemma_reduction_injective_on_eigenspace}. Moreover, we thank Bhargav Bhatt, Akhil Mathew, Andreas Mihatsch, Emmanuele Dotto, Matthew Morrow for answers and interesting related discussions. Very special thanks go to Irakli Patchkoria, who helped the authors with all necessary topology and provided detailed comments on a first draft. The authors thank an anonymous referee for her/his excellent report, in particular for suggesting the shorter proof of \Cref{proposition_second_description_of_dennis_trace}, which we presented here.
  
  The authors would like to thank the University of Bonn and the Institut de Math\'ematiques de Jussieu for their hospitality while this work was done.

\section{The $p$-completed Dennis trace in degree $2$}
\label{sec:p-completed-dennis}

Fix some prime $p$ and let $A=R/I$ be the quotient of a $(p,I)$-complete ring $R$. The aim of this section is to concretely describe in degree 2 the composition
$$
T_p(A^\times)\to
\pi_2(K(A;\Z_p)) \xrightarrow{\mathrm{Dtr}}\pi_2(\HH(A;\Z_p)) \to \pi_2(\HH(A/R;\Z_p)). 
$$
Here
$$
K(A;\Z_p)
$$
denotes the $p$-completed (connective) $K$-theory spectrum of $A$,
  $$
  \HH(A;\Z_p), \textrm{ resp. } \HH(A/R;\Z_p)
  $$
  the $p$-completed (derived) Hochschild homology of $A$ as a $\mathbb{Z}$-algebra, resp.\ as an $R$-algebra, and $\mathrm{Dtr}$ is the Dennis trace map. Before stating precisely our result, let us start by some reminders on the objects and the maps involved in the previous composition.
  \\ 

Let us first recall the construction of the first map $T_p(A^\times)\to \pi_2(K(A;\Z_p))$. Let
$$
\mathrm{GL}(A)=\varinjlim\limits_r \mathrm{GL}_r(A)
$$ 
be the infinite general linear group over $A$. There is a canonical
inclusion
$$
A^\times=\mathrm{GL}_1(A)\to \mathrm{GL}(A)
$$
of groups which on classifying spaces induces a map
$$
B(A^\times)\to B(\GL(A)).
$$
Composing with the morphism to Quillen's $+$-construction yields a
canonical morphism
$$
B(A^\times)\to B\GL(A)\to K(A):=B\GL(A)^+\times K_0(A)
$$
into the $K$-theory space $K(A)$ of $A$. After $p$-completion of spaces\footnote{We use space as a synonym for
  Kan complex.} we obtain a canonical morphism
$$
\iota\colon B(A^\times)^\wedge_p\to K(A;\Z_p):=K(A)^\wedge_p.
$$
We recall (cf.\ \cite[Theorem 10.3.2]{may_more_concise_algebraic_topology}) that the space $B(A^\times)^\wedge_p$ has two
non-trivial homotopy groups which are given by
$$
\pi_1(B(A^\times)^\wedge_p)\cong
H^0(R\varprojlim\limits_n(A^\times\otimes^{\mathbb{L}}_\Z \Z/p^n))
$$
and
$$
\pi_2(B(A^\times)^\wedge_p)\cong
H^{-1}(R\varprojlim\limits_n(A^\times\otimes^{\mathbb{L}}_\Z \Z/p^n))\cong
T_p(A^\times).
$$
In degree $2$ we thus get a morphism
$$
T_p(A^\times)=\pi_2(B(A^\times)^\wedge_p)\to \pi_2(K(A;\Z_p)),
$$
which is the first constituent of the map
$$
T_p(A^\times)\to
\pi_2(K(A;\Z_p))\xrightarrow{\mathrm{Dtr}}\pi_2(\HH(A;\Z_p)) \to \pi_2(\HH(A/R;\Z_p))
$$
we want to describe. 
\\

Now we turn to the construction of Hochschild homology and the Dennis
trace 
$$ 
\mathrm{Dtr}\colon K(A)\to \HH(A).
$$ 

Let $R$ be a (commutative) ring and $A$ a (commutative) $R$-algebra. Let 
$$
\mathbb{T}:=S^1\cong B\Z
$$
be the circle group. Then the Hochschild homology spectrum 
  $$
  \HH(A/R)
  $$
 (simply denoted $\HH(A)$ when $R=\Z$) is the initial $\mathbb{T}$-equivariant\footnote{For an $\infty$-category $\mathcal{C}$ the category of $\mathbb{T}$-equivariant objects of $\mathcal{C}$ is by definition the $\infty$-category of functors $B\mathbb{T}\to \mathcal{C}$.} $E_\infty-R$-algebra with a non-equivariant map $A\to \HH(A/R)$ of $E_\infty-R$-algebras, cf.\ \cite[Remark 2.4]{bhatt_morrow_scholze_topological_hochschild_homology}. For a comparison with classical definitions, we refer to \cite{hoyois_the_homotopy_fixed_points_of_the_circle_action_on_hochschild_homology}. 

  The functor $A\mapsto \HH(A/R)$ extends to all simplicial $R$-algebras and as such is left Kan extended (as it commutes with sifted colimits) from the category of finitely generated polynomial $R$-algebras. By left Kan extending the (decreasing) Postnikov filtration $\tau_{\geq \bullet}\HH(A/R)$ on $\HH(A/R)$ for $A$ a finitely generated polynomial $R$-algebra one obtains the $\mathbb{T}$-equivariant HKR-filtration
  $$
  \mathrm{Fil}^n_{\mathrm{HKR}}\HH(A/R)
  $$
  on $\HH(A/R)$ for $A$ a general $R$-algebra.
  The $\infty$-category of $\mathbb{T}$-equivariant objects in the derived $\infty$-category $\mathcal{D}(R)$ of $R$ is equivalent to the $\infty$-category of $R[\mathbb{T}]$-modules, where
  $$
  R[\mathbb{T}]=R\otimes \Sigma_+^\infty \mathbb{T}
  $$
  is the group algebra of $\mathbb{T}$ over $R$ (cf.\ \cite[page 5]{hoyois_the_homotopy_fixed_points_of_the_circle_action_on_hochschild_homology}).
  Let
  $$
  \gamma\in H_1(\mathbb{T},R)\cong \mathrm{Hom}_{\mathcal{D}(R)}(R[1],R[\mathbb{T}])
  $$
  be a generator\footnote{We will mostly assume that $\gamma$ is obtained by base change from some generator of $H_1(\mathbb{T},\Z)$.}. The multiplication by $\gamma$ induces a differential
  $$
  d\colon \HH_{i}(A/R)\to \HH_{i+1}(A/R)
  $$
  which makes $\HH_{\ast}(A/R)$ into a graded commutative dg-algebra over $R$ all of whose elements of odd degree square to zero (cf.\ \cite[Lemma 2.3]{nikolaus_krause_lectures_on_topological_hochschild_homology}).
  By the universal property of the de Rham complex $\Omega_{A/R}^\ast$, the canonical morphism $A\to \HH_0(A/R)$ extends therefore to a morphism
  $$
  \alpha_\gamma\colon \Omega_{A/R}^\ast\to \HH_{\ast}(A/R).
  $$
  The Hochschild-Kostant-Rosenberg theorem affirms that $\alpha_\gamma$ is an isomorphism if $R\to A$ is smooth. By left Kan extension, one obtains for arbitrary $R\to A$ the natural description
  $$
  \alpha_{\gamma}\colon \wedge^i L_{A/R}[i]\cong \mathrm{gr}^i_{\mathrm{HKR}} \HH(A/R)
  $$
  of the graded pieces of the HKR-filtration via exterior powers of the cotangent complex of $A$ over $R$ (cf. \cite[Section 2.2]{bhatt_morrow_scholze_topological_hochschild_homology}).
  
  In particular, we get after $p$-completion the following consequence in degree $2$, which will be used to formulate our description of the Dennis trace below.
  
  \begin{lemma}
  \label{lemma_second_hochschild_homology_for_a_surjection}
  Let $R$ be a ring and $I\subseteq R$ an ideal. Let $A=R/I$. Fix a generator $\gamma$ of $H_1(\mathbb{T},\Z)$. There is a
  natural isomorphism 
$$
\alpha_{\gamma} : (I/I^2)^\wedge_p\cong \pi_2(\HH(A/R;\Z_p)).
$$

\end{lemma}

  Here (and in the rest of the paper) we denote by $M^\wedge_p$ the \textit{derived} $p$-adic completion of an abelian group $M$, i.e.,
$$
M^\wedge_p:=H^0(R\varprojlim\limits_{n} M\otimes_{\Z}^{\mathbb{L}} \Z/p^n).
$$

\begin{proof}
  The first assertion follows from the HKR-filtration on $\HH(A/R;\Z_p)$
  described above and the fact there is a canonical isomorphism
$$
(I/I^2)^\wedge_p\cong H^{-1}((L_{A/R})^\wedge_p)
$$
which is implied by \cite[Tag 08RA]{stacks_project}.
\end{proof}

The Dennis trace can be defined abstractly, cf.\ \cite[Section 10.2]{blumberg_gepner_tabuada_a_universal_characterization_of_higher_algebraic_k_theory}, as the composition of the unique natural transformation
  $$
  K\to \THH
  $$ 
of additive invariants of small stable $\infty$-categories from $K$-theory to topological Hochschild homology, which induces the identity on 
$$
\Z\cong \pi_0(K(\mathbb{S}))\to \pi_0(\THH(\mathbb{S}))\cong \Z,
$$
 and the natural transformation\footnote{On rings.} $\THH\to \HH$.

The only thing we will need to use as an input regarding the Dennis trace is the following explicit description in degree $1$. Recall from above that if $A$ is a ring, each choice of a generator $\gamma$ of $H_1(\mathbb{T},\Z)$ gives rise to an isomorphism
$$
\alpha_{\gamma} : \pi_1(\HH(A/\Z))\cong \Omega^1_{A/\Z}
$$
as $H^0(L_{A/\Z})\cong \Omega^1_{A/\Z}$ for any $A$.

 \begin{lemma}
    \label{sec:p-completed-dennis-1-dennis-trace-in-degree-1}
    Let $A$ be a commutative ring. There exists a unique bijection
    $$
    \delta_1\colon \{ \textrm{generators of } H_1(\mathbb{T},\Z)\}\cong \{\pm 1\}
    $$
    such that
    $$
    A^\times\cong \pi_1(BA^\times)\overset{\mathrm{Dtr}}{\to} \pi_1(\HH(A))\overset{\alpha_\gamma}{\cong} \Omega^1_{A/\Z},\
a\mapsto \delta_1(\gamma)\mathrm{dlog}(a)
$$
for any generator $\gamma\in H_1(\mathbb{T},\Z)$.
  \end{lemma}
\begin{proof}
Let $A$ be any commutative ring. The Hochschild homology
$\HH(A)$ can be calculated as the geometric realization
$$
\HH(A):=\varinjlim\limits_{\Delta^{\mathrm{op}}}
A^{{\otimes^{\mathbb{L}}_{\Z} }^{n+1}}, 
$$
Note that this representation, which relies on the standard simplicial model of the circle $\Delta^1/\partial \Delta^1$, depends implicitly on the choice of a generator $\gamma_0$ of $H_1(\mathbb{T},\Z)$, cf. \cite[Theorem 2.3]{hoyois_the_homotopy_fixed_points_of_the_circle_action_on_hochschild_homology}\footnote{In this reference, $\gamma_0$ is called $\gamma$.}.
Replacing the derived tensor product by the non-derived one obtains the classical, non-derived Hochschild homology $\HH^{\mathrm{usual}}(A)$ of $A$. As 
$$
\pi_1(\HH(A))\cong \pi_1(\HH^{\mathrm{usual}}(A))
$$ 
we may argue using $\HH^{\mathrm{usual}}$ instead of $\HH$.

Using the above description of the classical Hochschild homology, the Dennis trace can be described more concretely, cf.\ \cite[Section 5]{boekstedt_hsiang_madsen_the_cyclotomic_trace_and_algebraic_k_theory_of_spaces}, \cite[Chapter 8.4]{loday_cyclic_homology}.  
It factors (on homotopy groups) through the integral
group homology of $\GL(A)$, i.e., through $H_\ast(B\GL(A),\Z)$, which
is by definition (and the Dold-Kan correspondence) the homotopy of the
space $\Z[B\GL(A)]$ obtained by taking the free simplicial abelian
group on the simplicial $B\GL(A)$. As the $+$-construction
$$
B\GL(A)\to B\GL(A)^+
$$
is an equivalence on integral homology (cf.\ \cite[Chapter IV.Theorem 1.5]{weibel_the_k_book}) the morphism
$$
\Z[B\GL(A)]\simeq \Z[B\GL(A)^+]
$$ 
is an equivalence of simplicial abelian groups and using the canonical
inclusion
$$
B\GL(A)^+\to \Z[B\GL(A)^+]
$$ 
we arrive at a canonical morphism
$$
K(A)\to B\GL(A)^+\to\Z[B\GL(A)^+]\simeq \Z[B\GL(A)].
$$

We observe that for $r=1$ the morphism $B\GL_1(A)\to B\GL_1(A)^+$ is
an equivalence as $\GL_1(A)=A^\times$ is abelian. Thus there is a
commutative diagram (up to homotopy)
$$
\xymatrix{
  B\GL_1(A)\ar[r]\ar[d] & \Z[B\GL_1(A)]\ar[d] \\
  K(A) \ar[r] & \Z[B\GL(A)] }
$$
with each morphism being the canonical one.

 The Dennis
trace factors as a composition
$$
\mathrm{Dtr}\colon K(A)\to
\Z[B\GL(A)]\xrightarrow{\mathrm{Dtr}^\prime}\HH^{\mathrm{usual}}(A/\Z),
$$
  where by
construction
$$
\mathrm{Dtr}^\prime\colon \Z[B\GL(A)]\to \HH^{\mathrm{usual}}(A)
$$
is given as the colimit of compatible maps\footnote{Here
  compatible means up to some homotopy. To obtain strict compatibility
  one has to use the normalised Hochschild complex, cf.\ \cite[Section 8.4.]{loday_cyclic_homology}}
$$
\mathrm{Dtr}^\prime_r\colon \Z[B\GL_r(A)]\to \HH^{\mathrm{usual}}(A).
$$

When $r=1$, which is the only case relevant for us, the map
$\mathrm{Dtr}^\prime_1$ is the linear extension of the map
$$
BA^\times\to \HH^{\mathrm{usual}}(A)
$$
which in simplicial degree $n$ is given by
$$
(a_1,\ldots, a_n)\mapsto \frac{1}{a_1\ldots a_n}\otimes a_1\otimes
\ldots \otimes a_n.
$$

Fix a generator $\gamma$ of $H_1(\mathbb{T},\Z)$. The choice of $\gamma$ gives the HKR-isomorphism
$$
\alpha_{\gamma}: \pi_1(\HH^{\mathrm{usual}}(A))\cong \pi_1(\HH(A/\Z))\cong \Omega^1_{A/\Z}.
$$
Using the above description of Hochschild homology as a geometric realization, the isomorphism $\alpha_{\gamma}$ is given by
$$
\pi_1(\HH^{\mathrm{usual}}(A))\cong \Omega^1_{A/\Z},\ a\otimes b\mapsto adb
$$
with inverse $adb\mapsto a\otimes b$, if $\gamma=\gamma_0$, and by
$$
\pi_1(\HH^{\mathrm{usual}}(A))\cong \Omega^1_{A/\Z},\ a\otimes b\mapsto bda
$$
with inverse $bda\mapsto a\otimes b$ if $\gamma=-\gamma_0$; this can be checked by analyzing compatibility with differentials and using \cite[Theorem 2.3]{hoyois_the_homotopy_fixed_points_of_the_circle_action_on_hochschild_homology}. In the first case, we set $\delta_1(\gamma)=1$; in the second case, we set $\delta_1(\gamma)=-1$. Then on homotopy groups the map
$\mathrm{Dtr}_1$ is given by
$$
A^\times\cong \pi_1(BA^\times)\to \pi_1(\HH(A)) \overset{\alpha_{\gamma}} \cong \Omega^1_{A/\Z},\
a\mapsto \delta_1(\gamma) \mathrm{dlog}(a):= \delta_1(\gamma) \frac{da}{a},
$$
as claimed.
\end{proof}

\begin{remark}
\label{old_result}
 Let $A$ be a flat $\Z$-algebra. The description of $\HH(A)=\HH^{\mathrm{usual}}(A)$ as the geometric realization of the simplicial object
$$
\HH(A/\Z):=\varinjlim\limits_{\Delta^{\mathrm{op}}}
A^{{\otimes_{\Z} }^{n+1}}
$$
shows that $\HH(A;\Z_p)$ is computed by the complex
$$
\dots \to (A \otimes_{\Z} A \otimes_{\Z} A)^{\wedge}_p \to  (A \otimes_{\Z} A)^{\wedge}_p \to A^{\wedge}_p.
$$
 One can then show that the $p$-completed Dennis trace
  $(BA^\times)^\wedge_p\to \HH(A;\Z_p)$
  sends an element
$$
(a_1, a_2,\ldots)\in T_p(A^\times)=\pi_2((BA^\times)^\wedge_p)
$$
to the element represented, up to a sign, by the cycle
$$
1\otimes 1\otimes 1+\sum\limits_{n=1}^{\infty} p^{n-1}(\frac{1}{a_n^2}\otimes a_n\otimes a_n+\frac{1}{a_n^3}\otimes a_n^2\otimes a_n+\ldots
+ \frac{1}{a_n^{p}}\otimes a_n^{{p-1}}\otimes a_n).
$$ 
We omit the proof, since we will not use this result.
\end{remark}

We can now state and prove the main result of this section. Fix a generator $\gamma$ of $H_1(\mathbb{T},\Z)$. We will describe the image of some element
$T_p(A^\times)$ under the composition
$$
T_p(A^\times)\xrightarrow{\mathrm{Dtr}} \pi_2(\HH(A;\Z_p))\to
\pi_2(\HH(A/R;\Z_p)) \overset{\alpha_{\gamma}^{-1}}\cong (I/I^2)^\wedge_p,
$$
using the notation of \Cref{lemma_second_hochschild_homology_for_a_surjection}. Recall first the following standard lemma.

\begin{lemma}
  \label{lemma_p_power_compatible_elements_in_the_reduction}
  Let $R$
  be a ring, $I\subseteq R$ an ideal and assume that $R$ is
  $(p,I)$-adically complete. Then the canonical map
$$
R^\flat:=\varprojlim\limits_{x\mapsto x^p} R\to A^\flat:=\varprojlim\limits_{x\mapsto
  x^p} A
$$
with $A=R/I$ is bijective.
\end{lemma}
\begin{proof}
  It suffices to construct a well-defined, multiplicative map
$$
[-] \colon A^\flat\to R
$$
reducing to the first projection modulo $I$.  Let
$$
r:=(r_0,r_1,\ldots)\in A^\flat
$$
be a
$p$-power compatible system of elements in $A$ with lifts
$r_i^\prime\in R$ of each $r_i$. Then the limit
$$
\lim\limits_{n\to \infty} (r_n^\prime)^{p^n}
$$
exists and is independent of the lift. Thus
$$
[r]:=\lim\limits_{n\to \infty} (r_n^\prime)^{p^n}
$$
defines the desired map.
\end{proof}

The morphism
$$
[-]\colon A^\flat\to R
$$
is the Teichm\"uller lift for the surjection $\pi\colon R\to R/I$. If we want to make its dependence of the surjection clear, we write $[-]_{\pi}$.
Let 
$$
T_pA^\times=\varprojlim\limits_{x\mapsto x^p}A^\times[p^n]
$$ 
be the Tate module of $A^\times$. Then we embed $T_pA^\times$ into $A^\flat$ as the sequences with first coordinate $1$.
For any $a\in A^\flat $ we define
$$
[a]:=r_0
$$
where $r=(r_0,r_1,\ldots)\in R^\flat$ is the unique element reducing to $a$. If $a=(1,a_1,a_2,\ldots)$ lies in $T_pA^\times$, then $[a]\in 1+I$.

\begin{proposition}
  \label{proposition_second_description_of_dennis_trace}
 Fix a generator $\gamma\in H_1(\mathbb{T},\Z)$. Let $R$ be a ring and $I\subseteq R$ an ideal such that $R$ is $(p,I)$-adically complete. Let
  $A=R/I$. Then the composition
$$
T_p(A^\times)\cong
\pi_2((BA^\times)^\wedge_p)\xrightarrow{\mathrm{Dtr}}
\pi_2(\HH(A/R;\Z_p))\cong (I/I^2)^\wedge_p
$$
is given by sending $a\in T_p(A^\times)$ to
$$
\delta_1(\gamma)([a]-1),
$$
 with $\delta_1(\gamma)\in \{\pm 1\}$ is the sign from \Cref{sec:p-completed-dennis-1-dennis-trace-in-degree-1}.
\end{proposition}

\begin{proof}\footnote{The following argument is simpler than our original argument and was suggested by the referee. We thank her/him for allowing us to include it.
}  Fix $a\in T_p(A^\times)$. Then there exists, by $(p,I)$-adic completeness of $R$, a unique morphism
  $\Z[1/p]\to R^\times$ of abelian groups such that
  $$
  1/p^n\mapsto [a^{1/p^n}].
  $$
    By naturality, it therefore suffices to check that for
    $$
    R:=\Z[t^{1/p^\infty}]\cong \Z[\Z[1/p]]
    $$
    and
    $$
    A:=R/(t-1)\cong \Z[\Q_p/\Z_p],
    $$
    then, under the morphism,
    $$
    T_pA^\times\xrightarrow{\mathrm{Dtr}} \HH_2(A;\Z_p)\to \HH_2(A/R;\Z_p)\cong L_{A/R}[-1]\cong (t-1)/(t-1)^2
    $$
    the element $(1,t^{1/p},t^{1/p^2},\ldots )\in A^\flat$ is mapped to the class of $\delta_1(\gamma)(t-1)$. This is what we will do.
    
    Observe first that the Hochschild homology
    $$
    \HH_2(A)
    $$
    vanishes. Indeed, it is easy to see that $L_{A/\Z}$ is concentrated in degree $0$. Moreover, $\Omega_{A/\Z}^1\cong L_{A/\Z}$ is generated by one element. This implies that
    $$
    \pi_0(\wedge^nL_{A/\Z})=0
    $$
    for $n\geq 2$ (cf.\ the proof of \cite[Corollary 3.13]{bhatt_p_adic_derived_de_rham_cohomology}). By the HKR-filtration, we get that $\HH_2(A)=0$.
    Passing to $p$-completions we can conclude that
    $$
    \HH_2(A;\Z_p)\cong T_p\HH_1(A)\overset{\alpha_\gamma}{\cong} T_p(\Omega^1_{A/\Z}),
    $$
where the last isomorphism is the HKR-isomorphism (for $\gamma$).
    
    There is a commutative diagram
    $$
    \xymatrix{
      \HH_2(A;\Z_p)\ar[r]^\cong\ar[d]^{\cong} & \HH_2(A/R;\Z_p)\ar[d]^{\cong} \\
      T_p\Omega_{A/\Z}^1\cong \pi_1((L_{A/\Z})^\wedge_p)\ar[r]^-\cong & \pi_1((L_{A/R})^\wedge_p)\cong ((t-1)/(t-1)^2)^\wedge_p. 
    }
    $$
    Using \Cref{sec:p-completed-dennis-1-dennis-trace-in-degree-1} the element
    $$
    (1,t^{1/p},t^{1/p^2},\ldots )\in T_pA^\times
    $$
    is mapped to the element
    $$
    \delta_1(\gamma)(0,d\log(t^{1/p}),d\log(t^{1/p^2}),\ldots)\in T_p(\Omega^1_{A/\Z}).
    $$
    The effect of the bottom row can be calculated using the exact triangle
    $$
    L_{R/\Z}\otimes^{\mathbb{L}}_{R} A\to L_{A/\Z}\xrightarrow{\beta} L_{A/R}
    $$
    and applying $p$-completions. More precisely, rotating plus the isomorphisms
    $$
    L_{R/\Z}\cong \Omega_{R/\Z}^1,\ L_{A/R}\cong (t-1)/(t-1)^2[1]
    $$
    yield the exact triangle
    $$
    (t-1)/(t-1)^2\xrightarrow{d} \Omega_{R/\Z}^1\otimes_{R} A\to \Omega_{A/R}^1\to (t-1)/(t-1)^2[1] 
    $$
    where the first morphism is the differential.
    Now apply (derived) $p$-completion to this exact triangle, the resulting connecting morphism
    $$
    T_p(\Omega^1_{A/\Z})\to (t-1)/(t-1)^2
    $$
    sends to $(0,d\log(t^{1/p}),d\log(t^{1/p^2}),\ldots )$ to $t-1$ as $t-1 \equiv \frac{t-1}{t}$ mod $(t-1)^2$ and
    $$
    \frac{d(t-1)}{t}=d\log(t)=p^nd\log(t^{1/p^n})
    $$
    for all $n\geq 0$\footnote{If $0\to M\to N\to Q\to 0$ is a short exact sequence of abelian groups, then the boundary map $T_pQ\to M^{\wedge}_p$ has the following description: take $x:=(q_i)_{i\geq 0}\in T_pQ$ and lift each $q_i$ to some $n_i\in N$. Then $p^in_i\in M$ and the limit $\varprojlim\limits p^in_i\in M^{\wedge}_p$ exists and is the image of $x$.}. Thus,
    $$
    \beta((0,d\log(t^{1/p}),d\log(t^{1/p^2}),\ldots)=t-1
    $$
    as claimed. 
\end{proof}

We recall the following lemma.  For a perfect ring $S$ we denote
its ring of Witt vectors by $W(S)$.

\begin{lemma}
  \label{lemma_p_completion_does_not_change_under_witt_vector_base}
  Let $S$ be a perfect ring and let $A$ be an $W(S)$-algebra. Then the
  canonical morphism
$$
\HH(A;\Z_p)\to \HH(A/W(S);\Z_p)
$$
is an equivalence.
\end{lemma}
\begin{proof}
  By the HKR-filtration, it suffices to see that the canonical
  morphism
$$
L_{A/\Z}\to L_{A/W(S)}
$$
of cotangent complexes is a $p$-adic equivalence, i.e., an equivalence
after $-\otimes^{\mathbb{L}}_\Z \Z/p$. Computing the right hand side
by polynomial algebras over $W(S)$ we see that it suffices to consider
the case that $A$ is $p$-torsion free. Then by base change
$$
L_{A/\Z}\otimes^{\mathbb{L}}_\Z \Z/p\cong L_{(A/p)/\F_p}
$$
resp.\
$$
L_{A/W(S)}\otimes^{\mathbb{L}}_\Z \Z/p\cong L_{(A/p)/S}
$$
and the claim follows from the transitivity triangle
$$
A/p\otimes^{\mathbb{L}}_S L_{S/\F_p}\to L_{(A/p)/\F_p}\to L_{(A/p)/S}
$$
using that $S$ is perfect which implies that the cotangent complex $L_{S/\F_p}$ of $S$ over $\F_p$ vanishes.
\end{proof}

\section{Transversal prisms}
\label{sec:some-results-prisms}

In this section we want to prove the crucial injectivity statement (\Cref{sec:transversal-prisms-corollary-injectivity-for-reduction-modulo-second-nygaard-filtration}) mentioned in the introduction.
Let us recall the following definition from \cite{bhatt_scholze_prisms_and_prismatic_cohomology}. 

\begin{definition}
A \textit{$\delta$-ring} is a pair $(A,\delta)$, where $A$ is a commutative ring, $\delta\colon A \to A$ a map of sets, with $\delta(0)=0$, $\delta(1)=0$, and
\[ \delta(x+y)=\delta(x)+\delta(y)+ \frac{x^p+y^p-(x+y)^p}{p} ~~ ; ~~ \delta(xy)=x^p \delta(y)+ y^p \delta(x) +p\delta(x)\delta(y), \]
for all $x,y \in A$. 

A \textit{prism} $(A,I)$ is a $\delta$-ring $A$ with an ideal $I$ defining a Cartier divisor on $\mathrm{Spec}(A)$, such that $A$ is derived $(p,I)$-adically complete and $p\in (I,\varphi(I))$.
\end{definition}

Here, the map
$$
\varphi\colon A\to A,\ x\mapsto \varphi(x):=x^p+p\delta(x)
$$
denotes the lift of Frobenius induced from $\delta$-structure on $A$. 
We will make the (usually harmless) assumption that $I=(\tilxi)$ is generated by some distinguished element $\tilxi\in A$, i.e., $\tilxi$ is a non-zero divisor and $\delta(\tilxi)$ is a unit.

\begin{definition}
  \label{sec:some-results-about-definition-transversal-prism}
  We call a prism \textit{transversal} if $(p,\tilxi)$ is a regular sequence on $A$.
\end{definition}

Let us fix a transversal prism $(A,I)$.
In particular, $A$ is $p$-torsion free.
Moreover, $A$ is classically $(p,I)$-adically complete. Indeed, $(p,\tilxi)$ being a regular sequence implies that
$$
A\otimes^{\mathbb{L}}_{\Z[x,y]}\Z[x,y]/(x^n,y^n)\cong A/(p^n,\tilxi^n)
$$
and therefore
$$
A\cong R\varprojlim\limits_n(A\otimes^{\mathbb{L}}_{\Z[x,y]}\Z[x,y]/(x^n,y^n)\cong R\varprojlim\limits_n(A/(p^n,\tilxi^n))\cong \varprojlim\limits_nA/(p^n,\tilxi^n)
$$
using Mittag-Leffler for the last isomorphism.

We set
$$
I_r:=I\varphi(I)\ldots \varphi^{r-1}(I)
$$
for $r\geq 1$ (where $\varphi^0(I):=I$).
Then $I_r=(\tilxi_r)$ with
$$
\tilxi_r:=\tilxi\varphi(\tilxi)\cdots \varphi^{r-1}(\tilxi).
$$
\begin{lemma}
  \label{lemma_p_torsionfree_implies_nonzerodivisor}
  For all $r\geq 1$ the element
  $$
  \varphi^r(\tilxi)
  $$
  is a non-zero divisor and $(\varphi^r(\tilxi),p)$ is again a
  regular sequence.  In particular, the elements $\tilxi_r$,
  $r\geq 1$, are non-zero divisors.
\end{lemma}
\begin{proof}
    The regularity of the sequence $(p,\varphi^r(\tilxi))$, or equivalently of $(p,\tilxi^{p^r})$, follows from the one of $(p,\tilxi)$. The regularity of $(\varphi(\tilxi^{p^r}),p)$ follows from this and the fact that in any ring $R$ with a regular sequence $(r,s)$ such that $R$ is $r$-adically complete the sequence $(s,r)$ is again regular\footnote{Passing to the inverse limit of the injections $R/r^{n}\xrightarrow{s}R/r^n$ implies that $s\in R$ is a non-zero divisor. Thus, $(r,s)$ is regular and $s$ is regular, which implies that $(s,r)$ is regular.}.
  \end{proof}

\begin{lemma}
  \label{lemma_i_r_completeness}
  The ring $A$ is complete for the topology induced by the ideals
  $I_r$, i.e.,
  $$
  A\cong \varprojlim\limits_{r}A/I_r.
  $$
\end{lemma}
\begin{proof}
    Each $A/I_r$ is $p$-torsion free by \Cref{lemma_p_torsionfree_implies_nonzerodivisor}. Therefore both sides are $p$-complete and $p$-torsion free. Hence, it suffices to check the statement modulo $p$ (note that by $p$-torsion free of each $A/I_r$ modding out $p$ commutes with the inverse limit). But modulo $p$ the topology defined by the ideals $I_r$ is just the $\tilxi$-adic topology and $A/p$ is $\tilxi$-adically complete.
\end{proof}

\begin{lemma}
  \label{lemma_reductions_of_phipowers_of_xi}
  For $r\geq 1$ there is a congruence
$$
\varphi^r(\tilxi)\equiv pu \textrm{ modulo } (\tilxi)
$$
with $u\in A^\times$ some unit.
\end{lemma}
\begin{proof}
  For $r=1$ this follows from
  $$
  \varphi(\tilxi)=\tilxi^p+p\delta(\xi)
$$
because by definition of distinguishedness the element
$\delta(\xi)\in A^\times$ is a unit.  For $r\geq 2$ we compute
$$
\varphi^r(\tilxi)=\varphi^{r-1}(\tilxi^p+p\delta(\tilxi))=\varphi^{r-1}(\tilxi)^p+p\varphi^{r-1}(\delta(\tilxi)).
$$
By induction we may write $\varphi^{r-1}(\tilxi)=pu+a\tilxi$
with $u\in A^\times$ some unit and thus modulo $\tilxi$ we
calculate
$$
\varphi^r(\tilxi)\equiv(pu)^p+p\varphi(\delta(\tilxi))=p(\varphi(\delta(\tilxi))+p^{p-1}u^p)
$$
with $\varphi(\delta(\tilxi))+p^{p-1}u^p\in A^\times$ some unit.
\end{proof}

\begin{lemma}
  \label{lemma_transversal_implies_regular_sequence}
  For all $r\geq 1$ the sequences
  $(\varphi^r(\tilxi),\tilxi)$ and
  $(\tilxi,\varphi^r(\tilxi))$ are again regular.  Moreover,
  $I_r=\bigcap\limits_{i=0}^{r-1}\varphi^i(I)$ for all $r\geq 1$.
\end{lemma}
\begin{proof}
  We can write
  $\varphi(\tilxi)=p\delta(\tilxi)+\tilxi^p$, where
  $\delta(\tilxi)\in A^\times$ is a unit. By
  \Cref{lemma_reductions_of_phipowers_of_xi} we get
  $\varphi^r(\tilxi)\equiv pu$ modulo $(\tilxi)$ with
  $u\in A^\times$ a unit. As $(\tilxi,p)$ is a regular sequence
  we conclude (using \cite[Tag 07DW]{stacks_project} and
  \Cref{lemma_p_torsionfree_implies_nonzerodivisor}) that
  $(\varphi^r(\tilxi),\tilxi)$ is a regular sequence.  To
  prove the last statement we proceed by induction on $r$.  First note
  the following general observation: If $R$ is some ring and $(f,g)$ a
  regular sequence in $R$, then $(f)\cap (g)=(fg)$. In fact, if
  $r=sg\in (f)\cap (g)$, then $sg\equiv 0$ modulo $f$, hence
  $s\equiv 0$ modulo $f$ as desired.  Thus, it suffices to prove that
  $(\tilxi_r,\varphi^r(\tilxi))$ is a regular sequence for
  $r\geq 1$ (recall that
  $\tilxi_r=\tilxi\varphi(\tilxi)\cdots
  \varphi^{r-1}(\tilxi)$). By induction the morphism
  $$
  A/(\tilxi_r)\to
  \prod\limits_{i=0}^{r-1}A/(\varphi^i(\tilxi))
$$
is injective. Hence, it suffices to show that for each
$i=0,\ldots,r-1$ the element $\varphi^r(\tilxi)$ maps to a
non-zero divisor in $A/(\varphi^i(\tilxi))$. But this follows
from \Cref{lemma_reductions_of_phipowers_of_xi} which implies
$\varphi^r(\tilxi)\equiv pu$ modulo $\varphi^i(\tilxi)$ for
some unit $u\in A^\times$.
\end{proof}

We can draw the following corollary.

\begin{lemma}
  \label{lemma_injectivity_for_transversal}
  Define
  $\rho\colon A\to \prod\limits_{r\geq 0} A/\varphi^r(I),\ x\mapsto
  (x\ \mathrm{ mod }\ \varphi^r(I))$. Then $\rho$ is injective.
\end{lemma}
\begin{proof}
  This follows from \Cref{lemma_i_r_completeness} and
  \Cref{lemma_transversal_implies_regular_sequence} as the kernel of
  $\rho$ is given by
  $\bigcap\limits_{r=1}^\infty
  \varphi^r(I)=\bigcap\limits_{r=1}^\infty I_r=0$.
\end{proof}

We now define the Nygaard filtration of the prism $(A,I)$ (cf.\ \cite[Definition 11.1]{bhatt_scholze_prisms_and_prismatic_cohomology}).

\begin{definition}
  \label{sec:transversal-prisms-definition-nygaard-filatration}
Define
$$
\mathcal{N}^{\geq n}A:=\{ x\in A\ |\ \varphi(x)\in I^nA\},
$$
the $n$-th filtration step of the Nygaard filtration.
\end{definition}

By definition the Frobenius on $A$ induces a morphism
$$
\varphi\colon \mathcal{N}^{\geq 1}A \to I.
$$
Note that we do not divide the Frobenius by $\tilxi$.  Moreover,
we define
$$
\sigma\colon \prod\limits_{i\geq 0} A/\varphi^i(I)\to
\prod\limits_{i\geq 0} A/\varphi^i(I),\ (x_0,x_1,\ldots)\mapsto
(0,\varphi(x_0),\varphi(x_1),\ldots).
$$
Here we use that if $a\equiv b$ mod $\varphi^i(I)$, then
$\varphi(a)\equiv \varphi(b)$ mod $\varphi^{i+1}(I)$ to get that
$\sigma$ is well-defined.  Then the diagram
\begin{equation}
  \label{equation_commutative_square_before_injectivity_statement}
\xymatrix{
  \mathcal{N}^{\geq 1}A \ar[r]^-\rho\ar[d]^\varphi & \prod\limits_{i\geq 0} A/\varphi^i(I)\ar[d]^\sigma \\
  I\ar[r]^-\rho & \prod\limits_{i\geq 0} A/\varphi^i(I) }
\end{equation}
commutes where $\rho$ is the homomorphism from \Cref{lemma_injectivity_for_transversal}.

\begin{lemma}
  \label{lemma_reduction_injective_on_eigenspace}
  The reduction map
  $$
  A^{\varphi=\tilxi}\to A/I,\ x\mapsto x\
  \mathrm{mod}\ (\tilxi)
  $$
  is injective.
\end{lemma}
\begin{proof}
  Let $x\in A^{\varphi=\tilxi}\cap I$. We
  want to prove that $x=0$. Clearly, $x\in \mathcal{N}^{\geq 1}A$. By \Cref{lemma_injectivity_for_transversal}
  it suffices to prove that
  $$
  x \equiv 0\ \mathrm{mod }\ \varphi^i(I)
$$
for all $i\geq 0$.  Write
$$
\rho(x)=(x_0,x_1,\ldots)
$$
By the commutativity of the square (\Cref{equation_commutative_square_before_injectivity_statement}) we get
$$
\rho(\varphi(x))=\sigma(\rho(x))=(0,\varphi(x_0),\varphi(x_1),\ldots).
$$
As $\varphi(x)=\tilxi x$ and therefore
$\rho(\varphi(x))=\tilxi\rho(x)$ we thus get
$$
(\tilxi x_0,\tilxi x_1,\ldots)=(0,\varphi(x_0),\varphi(x_1),\ldots).
$$
We assumed that $x\in I$, thus $x_0=0\in A/I$. Now we use that
$\tilxi$ is a non-zero divisor modulo $\varphi^i(I)$
(cf. \Cref{lemma_transversal_implies_regular_sequence}) for
$i>0$. Hence, if $x_i=0$, then
$$
0=\varphi(x_i)=\tilxi x_{i+1}\in A/\varphi^{i+1}(I)
$$
implies $x_{i+1}=0$. Beginning with $x_0=0$ this shows that $x_i=0$
for all $i \geq 0$, which implies our claim.
\end{proof}

The same proof shows that also the reduction map
$$
A^{\varphi=\tilxi^n}\to A/I
$$
is injective for $n\geq 1$.

The following corollary is crucially used in \Cref{sec:conclusion-theorem-identification-of-cyclotomic-trace}.

\begin{corollary}
  \label{sec:transversal-prisms-corollary-injectivity-for-reduction-modulo-second-nygaard-filtration}
  The reduction map
  $$
  A^{\varphi=\tilxi}\to \mathcal{N}^{\geq 1}A/\mathcal{N}^{\geq 2}A
  $$
  is injective.
\end{corollary}
\begin{proof}
  Let $x\in A^{\varphi=\tilxi}\cap \mathcal{N}^{\geq 2} A$. Then
  $$
  \tilxi x=\varphi(x)=\tilxi^{2}y
  $$
  for some $y\in A$. As $\tilxi$ is a non-zero divisor in $A$ we get $x\in I=(\tilxi)$. But then $x=0$ by \Cref{lemma_reduction_injective_on_eigenspace}.  
\end{proof}

  Similarly, for each $n\geq 0$ the morphism
  \begin{equation}
    \label{eq:1}
  A^{\varphi=\tilxi^n}\to \mathcal{N}^{\geq i}A/\mathcal{N}^{\geq i+1}A  
  \end{equation}
  is injective.
  Let $R$ be a quasi-regular semiperfectoid ring (cf.\ \cite[Definition 4.19]{bhatt_morrow_scholze_topological_hochschild_homology}) which is $p$-torsion free. In this case,
  $$
  A:=\widehat{\prism}_R
  $$
  is transversal and (\Cref{eq:1}) implies that for $i\geq 0$ 
  $$
  \pi_{2i}(\TC(R))\to \pi_{2i}(\THH(R))
  $$
  is injective (cf.\ \cite[Theorem 1.12]{bhatt_morrow_scholze_topological_hochschild_homology}). We ignore if there exists a direct topological proof. Note that the $p$-torsion freeness is necessary. Indeed, by \cite[Remark 7.20]{bhatt_morrow_scholze_topological_hochschild_homology} $\pi_{2i}(\TC(R))$ is always $p$-torsion free.

\section{The $q$-logarithm}
\label{section-q-logarithm}

In this section we recall the definition of the $q$-logarithm and prove some properties of it. \blue{Several statements in $q$-mathematics that we use are probably standard; cf. e.g. \cite{scholze_canonical_q_deformations_in_arithmetic_geometry} for more on $q$-mathematics.} Recall that the $q$-analog of the integer $n\in \Z$ is defined
to be
$$
[n]_q:=\frac{q^n-1}{q-1}\in \Z[q^{\pm 1}].
$$
If $n\geq 1$, then we can rewrite
$$
[n]_q=1+q+\ldots+ q^{n-1}
$$
and then the $q$-number actually lies in $\Z[q]$. For $n\geq 0$ we
moreover get the relation
\begin{equation}
  \label{eq:2-negative-q-numbers}
[-n]_q=\frac{q^{-n}-1}{q-1}=q^{-n}\frac{1-q^n}{q-1}=-q^{-n}[n]_q.  
\end{equation}
The $q$-numbers satisfy some basic relations, for example
\begin{equation}
  \label{eq:1-addition-for-q-numbers}
[n+k]_q=q^k[n]_q+[k]_q  
\end{equation}
for $n,k\in \Z$, or
$$
[m]_q=\frac{(q^n)^k-1}{q^n-1}\frac{q^n-1}{q-1}=\frac{(q^n)^k-1}{q^n-1}[n]_q,
$$
if $n|m$.
As further examples of $q$-analogs let us define the $q$-factorial for $n\geq 1$ as
$$
[n]_q!:=[1]_q\cdot [2]_q\cdot\ldots \cdot [n]_q\in \Z[q]
$$
(with the convention that $[0]_q!:=1$) and, for $0\leq k\leq n$, the
$q$-binomial coefficient as
$$
\binom{n}{k}_q:=\frac{[n]_q!}{[k]_q![n-k]_q!}.
$$

\begin{lemma}
  \label{lemma_q_binomial_is_integral}
  \begin{itemize}
  \item[1)] For $0\leq k\leq n$ the $q$-binomial
    $\binom{n}{k}_q\in \Z[q]$.
  \item[2)] For $1\leq k\leq n$ the analog
  $$
  \binom{n}{k}_q=q^k\binom{n-1}{k}_q+\binom{n-1}{k-1}_q
  $$
  of Pascal's identity holds.
\end{itemize}
\end{lemma}
\begin{proof}
  1) follows from 2) using induction and the easy case
  $\binom{n}{0}_q=1$. Then 2) can be proved as follows: Let
  $1\leq k\leq n$, then
  $$
  \begin{matrix}
    q^k\binom{n-1}{k}_q+\binom{n-1}{k-1}_q & = & \frac{[n-1]_q!}{[k-1]_q![n-1-k]_q!}(\frac{q^k}{[k]_q}+\frac{1}{[n-k]_q}) \\
    & = & \frac{[n-1]_q!}{[k-1]_q![n-1-k]_q!}(\frac{q^k[n-k]_q+[k]_q}{[k]_q[n-k]_q} \\
    & = & \frac{[n-1]_q!}{[k-1]_q![n-1-k]_q!}\frac{[n]_q}{[k]_q[n-k]_q} \\\
    & = & \binom{n}{k}_q
  \end{matrix}
  $$
  using the addition rule (\Cref{eq:1-addition-for-q-numbers}).
\end{proof}

Let us define a generalized $q$-Pochhammer symbol by
$$
(x,y;q)_n:=(x+y)(x+yq)\ldots (x+yq^{n-1})\in \Z[q^{\pm 1} ,x,y]
$$
for $n\geq 1$
(setting $x=1$ and $y:=-a$ recovers the known $q$-Pochhammer symbol
$$
(a;q)_n=(1-a)(1-aq)\ldots(1-aq^{n-1})=(1,-a;q)_n).
$$
Moreover we make the convention
$$
(x,y;q)_0:=1.
$$
In the $q$-world the generalized $q$-Pochhammer symbol replaces the polynomial
$$
(x+y)^n.
$$
For example one can show (using \Cref{lemma_q_binomial_is_integral}) the following $q$-binomial formula

\begin{equation}
  \label{eq:3-q-binomial-theorem}
  (x,y;q)_n=\sum\limits_{k=0}^n q^{k(k-1)/2}\binom{n}{k}_qx^{n-k}y^k.
\end{equation}
  
Let us now come to $q$-derivations.  We recall that the $q$-derivative
$\nabla_qf$ of some polynomial $f\in \Z[q^{\pm 1}][x^{\pm 1}]$ is defined by
$$
\nabla_qf(x):=\frac{f(qx)-f(x)}{qx-x}\in \Z[q^{\pm 1}][x^{\pm 1}].
$$
Thus for example, if $f(x)=x^n$, $n\in\Z$, then we can calculate
$$
\nabla_q(x^n)=\frac{q^nx^n-qx}{qx-x}=\frac{q^n-1}{q-1}x^{n-1}=[n]_qx^{n-1}.
$$
The $q$-derivative satisfies an analog of the Leibniz rule, namely
$$
\nabla_q(f(x)g(x))=\nabla_q(f(x))g(qx)+f(x)\nabla_q(g(x)).
$$

Similarly to the classical rule
$$
\nabla_x((x+y)^n)=n\nabla_x((x+y)^{n-1})
$$
we obtain the following relation for the generalized $q$-Pochhammer symbol.

\begin{lemma}
  \label{lemma_q_derivative_for_formal_powers}
  Let $\nabla_q:=\nabla_{q,x}$ denote the $q$-derivative with respect to $x$. Then the formula
  $$
  \nabla_{q}((x,y;q)_n)=[n]_q(x,y;q)_{n-1}
  $$
  holds in $\Z[q^{\pm 1}][x^{\pm 1},y^{\pm 1}]$.
\end{lemma}
\begin{proof}
  We proceed by induction on $n$. Let $n=1$. Then $(x,y;q)_n=x+y$
  and
  $$
  \nabla_q((x+y))=1.
  $$
  Now let $n\geq 2$. We calculate using induction
  $$
  \begin{matrix}
    \nabla_q((x,y;q)_n)& = & \nabla_q((x,y;q)_{n-1}(x+yq^{n-1})) \\
    & = & (x,y,q)_{n-1}\nabla_q(x+yq^{n-1})+(qx+q^{n-1}y)\nabla_q((x,y;q)_{n-1}) \\
    & = & (x,y;q)_{n-1}\cdot 1 + q(x+q^{n-2}y)[n-1]_q(x,y;q)_{n-2}\\
    & = & (1+q[n-1]_q)(x,y;q)_{n-1}\\
    & = & [n]_q (x,y;q)_{n-1}
  \end{matrix}
  $$
  where we used the $q$-Leibniz rule and (\Cref{eq:1-addition-for-q-numbers}).
\end{proof}

Similarly as the polynomials
$$
1, x-1, \frac{(x-1)^2}{2!},\ldots, \frac{(x-1)^n}{n!},\ldots
$$
are useful for developing some function into a Taylor series around
$x=1$ (because the derivative of one polynomial is the previous one)
the $q$-polynomials
$$
1,(x,-1;q)_1,\frac{(x,-1;q)_2}{[2]_q!},\ldots,
\frac{(x,-1;q)_n}{[n]_q!},\ldots
$$
are useful for developing a $q$-polynomial into some ``$q$-Taylor
series'' around $x=1$.  However, for this to make sense we have to pass to suitable completions and localize at $\{[n]_q\}_{n\geq 1}$. Let us be more precise about this. The $(q-1,x-1)$-completion $\Z[[q-1,x-1]]$ of $\Z[q,x]$ contains expressions of the form
$$
\sum\limits_{n=0}^\infty a_n(x,-1;q)_n
$$
with $a_n\in \Z[[q-1]]$ because
$$
(x,-1;q)_n=(x-1)(x-1+1-q)\ldots (x-1+(1-q)\frac{1-q^{n-1}}{1-q})\in (q-1,x-1)^n.
$$
Finally, the next calculations will take place in the ring\footnote{Note that inverting $[n]_q$ for $n\geq 0$ and then $q-1$-adically completing is the same as inverting $n$ for $n\geq 0$ and then $q-1$-adically completing.}
$$
\Q[[q-1,x-1]]\cong \Z[[q-1,x-1]][1/[n]_q|n\geq 1]^{\wedge}_{(q-1,x-1)}
$$
because
$$
\frac{(x,-1;q)_n}{[n]_q!}\in (q-1,x-1)_{\Q[[q-1,x-1]]}.
$$
The ring $\Q[[q-1,x-1]]$ admits a surjection to
$$
\Q[[q-1,x-1]]\to \Q[[x-1]]
$$
with kernel generated by $q-1$.
Similarly, there is a morphism
$$
\mathrm{ev}_1\colon \Q[[q-1,x-1]]\to \Q[[q-1]]
$$
with kernel generated by $x-1$.
Finally, the $q$-derivative $\nabla_q$ extends to a $q$-derivation on $\Q[[q-1,x-1]]$ and it induces the usual derivative after modding out $q-1$.
We denote by $\nabla_q^n$ the $n$-fold decomposition of $\nabla_q$ and by
$$
f(x)_{|x=1}:=\mathrm{ev}_1(f(x))
$$
the evaluation at $x=1$ of an element $f\in \Q[[q-1,x-1]]$.

\begin{lemma}
  \label{lemma_function_zero_if_q_derivatives_vanish}
  Let $f(x)\in \Q[[q-1,x-1]]$. If $\nabla_q^n(f(x))_{|x=1}=0$ for
  all $n\geq 0$, then $f(x)=0$.
\end{lemma}
\begin{proof}
  As $\nabla_q$ reduces to the usual derivative modulo $q-1$, we see
  that $f$ must be divisible by $q-1$, i.e., we can write
  $f(x)=(q-1)g(x)$ with $g(x)\in \Q[[q-1,x-1]]$. But then
  $\nabla_q^n(g(x))_{x=1}=0$ for all $n\geq 0$ and we can conclude as
  before that $q-1|g(x)$ which in the end implies
  $$
  f(x)\in \bigcap\limits_{k=1}^\infty (q-1)^k=\{0\}
  $$
  because $\Q[[q-1,x-1]]$ is $(q-1)$-adically separated.
\end{proof}

Now we can state the $q$-Taylor expansion around $x=1$ for elements in $\Q[[q-1,x-1]]$.

\begin{proposition}
  \label{proposition_q_taylor_expansion}
  For any $f(x)\in \Q[[q-1,x-1]]$ there is the Taylor expansion
  $$
  f(x)=\sum\limits_{n=0}^\infty \nabla_q^n(f(x))_{|x=1}
  \frac{(x,-1;q)_n}{[n]_q!}.
  $$
\end{proposition}
\begin{proof}
  Because
  $$
  \nabla_q(\frac{(x,-1;q)_n}{[n]_q!})=\frac{(x,-1;q)_{n-1}}{[n-1]_q!}
  $$
  we can directly calculate that both sides have equal higher
  derivatives at $x=1$. Thus they agree by
  \Cref{lemma_function_zero_if_q_derivatives_vanish}.
\end{proof}

Using this we can in \Cref{lemma_properties_q_logarithm} motivate the following formula for the $q$-logarithm.

\begin{definition}
  \label{definition_q_logarithm}
  We define the $q$-logarithm as
  $$
  \log_q(x):=\sum\limits_{n=1}^\infty
  (-1)^{n-1}q^{-n(n-1)/2}\frac{(x,-1;q)_n}{[n]_q}\in \Q[[q-1,x-1]].
$$

\end{definition}

Note that the element $\log_q(x)$ is contained in a much smaller subring of $\Q[[q-1,x-1]]$, it suffices to adjoin the elements $\frac{(x,-1;q)_n}{[n]_q}$ for $n\geq 0$ to $\Z[q^{\pm 1},x^{\pm 1}]$ and $(x-1)$-adically complete.

In the ring $\Q[[q-1,x-1]]$ the element $x$ is invertible, as
$$
\frac{1}{x}=\frac{1}{1-(1-x)}=1+(1-x)+(1-x)^2+\ldots\ .
$$

The $q$-derivative of the $q$-logarithm is $1/x$, similarly to the usual logarithm. 

\begin{lemma}
  \label{lemma_properties_q_logarithm}
  The $q$-logarithm $\log_q(x)$ is the unique $f(x)\in \Q[[q-1,x-1]]$ satisfying
  $f(1)=0$ and $\nabla_q(f(x))=\frac{1}{x}$.  Moreover,
  $$
  \log_q(x)=\frac{q-1}{\log(q)}\log(x)
  $$
 as elements in $\Q[[q-1,x-1]]$.
  
\end{lemma}
\begin{proof}
  That $\log_q(x)$ has $q$-derivative $1/x$ can be checked using \Cref{proposition_q_taylor_expansion} after writing $1/x$ in its $q$-Taylor expansion. Moreover, $\log_q(1)=0$. For the converse pick $f$ as in the statement.
  By \Cref{proposition_q_taylor_expansion} we can write
  $$
  f(x)=\sum\limits_{n=0}^\infty \nabla_q^n(f(x))_{|x=1}
  \frac{(x,-1;q)_n}{[n]_q!}.
  $$
  and thus we have to determine
  $$
  a_n:=\nabla_q^n(f(x))_{|x=1}
  $$
  for $n\geq 0$.  By assumption we must have $a_0=f(1)=0$. Moreover,
  for $n\geq 1$
  $$
  a_n=\nabla_q^n(f(x))_{|x=1}=\nabla^{n-1}_q(x^{-1})_{|x=1}=[-n+1]_q\ldots
  [-1]_q.
  $$
  Using $[-k]_q=-q^{-k}[k]_q$ for $k\in \Z$ the last expression
  simplifies to
  $$
  [-n+1]_q\ldots[-1]_q=(-1)^{n-1}q^{-n(n-1)/2}[n-1]_q!.
  $$
  Thus we can conclude
  $$
  f(x)=\sum\limits_{n=1}^\infty
  (-1)^{n-1}q^{-n(n-1)/2}\frac{(x,-1;q)_n}{[n]_q}=\log_q(x).
  $$
  For the last statement note that
  $$
  f(x):=\frac{q-1}{\log(q)}\log(x)
  $$
  exists in $\Q[[q-1,x-1]]$ (because $n\in R_{q}^\times$ for all $n\geq 1$) and
  satisfies $f(1)=0$. Moreover,
  $$
  \nabla_q(f(x))=\frac{f(qx)-f(x)}{qx-x}=\frac{q-1}{\log(q)}\frac{\log(q)+\log(x)-\log(x)}{(q-1)x}=\frac{1}{x}
  $$
  which implies $f(x)=\log_q(x)$ by the proven uniqueness of the
  $q$-logarithm.
\end{proof}

We now turn to prisms again.
Define
$$
\tilxi:=[p]_q=1+q+\ldots+q^{p-1}
$$
and
$$
\tilxi_r=\tilxi\varphi(\tilxi)\ldots
\varphi^{r-1}(\tilxi)
$$
for $r\geq 1$. Here, $\varphi$ is the Frobenius lift on $\Z[q^{\pm 1}]$
satisfying $\varphi(q)=q^p$.
Then $\tilxi$ is a distinguished element in the prism $\Z_p[[q-1]]$.
The $\tilxi_r$ are again $q$-numbers, namely
$$
\tilxi_r=[p^r]_q.
$$

Let us
recall the following situation from crystalline cohomology. Assume
that $A$ is a $p$-complete ring with an ideal $J\subseteq A$ equipped
with divided powers
$$
\gamma_n\colon J\to J,\ n\geq 1.
$$
In this situation the
logarithm
$$
\log(x):=\sum\limits_{n=1}^\infty (-1)^{n-1}(n-1)!\gamma_n(x-1)
$$
converges in $A$ for every element $x\in 1+J$. We now want
to prove an analogous statement for the $q$-logarithm.  Recall that for a prism $(A,I)$ we defined the Nygaard filtration
$$
\mathcal{N}^{\geq n}A:=\{ x\in A\ |\ \varphi(x)\in I^n \},\ n\geq 0
$$
in \Cref{sec:transversal-prisms-definition-nygaard-filatration}.
From now on, we assume that the prism $(A,I)$ lives over $(\Z_{p}[[q-1]],(\tilxi))$.
The expression
$$
\gamma_{n,q}(x-y):=\frac{(x-y)(x-qy)\cdots (x-q^{n-1}y)}{[n]_q!} \in \Z_p[[q-1]][x,y][1/{[m]_q}|\ m\geq 0]
$$
is called the $n$-th $q$-divided power of $x-y$ (cf.\ \cite[Rem. 1.4]{pridham_qderham})\footnote{This terminology is, however, quite bad. The $q$-divided power depends on the pair $(x,y)$ and not simply their difference $x-y$.}. We will study the
divisibility of
$$
(x-y)(x-qy)\cdots (x-q^{n-1}y)
$$
by
$$
\tilxi,\varphi(\tilxi),\ldots.
$$
The following statement is clear.

\begin{lemma}
  \label{sec:q-logarithm-lemma-varphi-tilde-xi-is-minimal-polynomial-of-root-of-unity}
  For $r\geq 1$ the polynomial (in $q$)
$$
\varphi^{r-1}(\tilxi)=\frac{q^{p^{r}}-1}{q^{p^{r-1}}-1}
$$
is the minimal polynomial of a $p^{r}$-th root of unity
$\zeta_{p^{r}}$, i.e., the morphism
$$
\Z[q]/{(\varphi^{r-1}(\tilxi))}\to \Z[\zeta_{p^{r}}],\ q\mapsto
\zeta_{p^{r}}
$$
is injective.
\end{lemma}

Thus reducing modulo $\varphi^{r-1}(\tilxi)$ is the same as setting
$q=\zeta_{p^{r}}$. Moreover, in $\Z[\zeta_{p^{r}}]$ there is the
equality
$$
z^{p^{r}}-1=\prod\limits_{i=0}^{p^{r}-1}(z-\zeta_{p^{r}}^i).
$$
Setting $z=\frac{x}{y}$ one thus arrives at the congruence

\begin{equation}
  \label{eq:congruence-for-q-divided-powers}
  x^{p^{r}}-y^{p^{r}}\equiv (x-y)(x-qy)\cdots (x-q^{p^{r}-1}y)\ \mathrm{ mod }\ \varphi^{r-1}(\tilxi),
\end{equation}
which will be useful.

\begin{lemma}
  \label{sec:q-logarithm-lemma-xi-valuation-of-q-factorial}
  Let $n\geq 1$ and for $r\geq 1$ write $n=a_rp^r+b_r$ with
  $a_r,b_r\geq 0$ and $b_r<p^r$. Then in $\Z_p[[q-1]]$
$$
[n]_q!=u\prod\limits_{r\geq 1}^\infty \varphi^{r-1}(\tilxi)^{a_r}
$$
for some unit $u\in \Z_p[[q-1]]^\times$.
\end{lemma}
\begin{proof}
  We may prove the statement by induction on $n$. Thus let us assume that it
  is true for $m=n-1$ and for $r\geq 1$ write $m=c_rp^r+d_r$ with
  $c_r,d_r\geq 0$ and $d_r<p^r$. If $n$ is prime to $p$, then $[n]_q$
  is a unit in $\Z_p[[q-1]]$ and it suffices to see that the righthand
  side is equal (up to some unit in $\Z_p[[q-1]]$) to
$$
\prod\limits_{r\geq 1}^\infty \varphi^{r-1}(\tilxi)^{c_r}.
$$ 
But $n$ prime to $p$ implies that $b_r>0$ for all $r\geq 1$. Thus
$c_r=a_r$ and $d_r=b_r-1$, which implies that both products are equal.
Now assume that $p$ divides $n$ and write $n=p^sn^\prime$ with
$n^\prime$ prime to $p$. Moreover, write $m=n-1=c_rp^r+d_r$ as
above. Then we can conclude $a_r=c_r$ for $r>s$ while $c_r=a_r-1$
for $1\leq r\leq s$ (as $d_r=p^r-1$ for such $r$). Altogether we
therefore arrive at
$$
\begin{matrix}
  [n]_q!=[n]_q[n-1]_q!\\
  = u^\prime [n]_q\prod\limits_{r\geq 1}^\infty \varphi^{r-1}(\tilxi)^{c_r}\\
  = u^\prime v\prod\limits_{r\geq 1}^\infty
  \varphi^{r-1}(\tilxi)^{a_r}
\end{matrix},
$$
$u^\prime\in \Z_p[[q-1]]^\times$, where we used that
$$
[n]_q=v[p^s]_q=v\varphi^{s-1}(\tilxi)\ldots \tilxi
$$
for some unit $v\in \Z_p[[q-1]]$.
\end{proof}

\begin{proposition}
  \label{sec:q-logarithm-proposition-q-divided-powers-for-rank-1-elements}
  Let $(A,I)$ be a prism over $(\Z_p[[q-1]],(\tilxi))$ and let
  $x,y\in A$ be elements of rank $1$ such that
  ${\varphi(x-y)}=x^p-y^p\in \tilxi A$. Then for all $n\geq 1$ the
  ring $A$ contains a $q$-divided power
$$
\gamma_{n,q}(x-y)=\frac{(x-y)(x-qy)\cdots (x-q^{n-1}y)}{[n]_q!}
$$   
of $x-y$\footnote{By this we mean that there exists an element, called $\gamma_{n,q}(x-y)$, such that $[n]_q!\gamma_{n,q}(x-y)=(x-y)(x-qy)\cdots (x-q^{n-1}y)$. The element $\gamma_{n,q}(x-y)$ need not be unique, but it is if $A$ is $[n]_q$-torsion free for any $n\geq 0$. Note that even in this torsion free case $\gamma_{n,q}(x-y)$ depends on the pair $(x,y)$ and not merely on the difference $x-y$.}. Moreover, $\gamma_{n,q}$ lies in fact in the $n$-th step
$\mathcal{N}^{\geq n} A$ of the Nygaard filtration of $A$.
\end{proposition}
\begin{proof}
  Replacing $A, x,y$ by the universal case we may assume that $A$ is
  flat over $\Z_p[[q-1]]$. In particular, this implies that
  $\tilxi,\varphi(\tilxi),\ldots$ are pairwise regular sequences (cf.\
  \Cref{lemma_transversal_implies_regular_sequence}).  Fix $n\geq 1$. For $r\geq 1$ we write $n$ as
$$
n=a_r p^r+b^r
$$ 
with $a_r,b_r\geq 0$ and $0\leq b^r<p^r$.  We claim that for each
$r\geq 0$
$$
\varphi^{r-1}(\tilxi)^{a_r}
$$
divides
$$
(x-y)(x-qy)\cdots (x-q^{n-1}y).
$$
This implies the proposition, namely by
\Cref{sec:q-logarithm-lemma-xi-valuation-of-q-factorial} we have
$$
[n]_q!=u\prod\limits_{r\geq 1} \varphi^{r-1}(\tilxi)^{a_r}
$$ 
for some unit $u\in A^\times$ while furthermore the morphism
$$
A/{([n]_q!)}\to \prod\limits_{r\geq 1}
A/{(\varphi^{r-1}(\tilxi))^{a_r}}
$$
is injective by the proof of \Cref{lemma_transversal_implies_regular_sequence}. Thus fix $r\geq 1$. To prove our claim we may
replace $n$ by $n-b_r$ as
$$
(x-y)(x-qy)\cdots (x-q^{n-b_r-1}y)
$$
divides
$$
(x-y)(x-qy)\cdots (x-q^{n-1}y).
$$
Thus let us assume that $n=a_rp^{r}$.  We claim that each of the
following $a_r$ many elements (note that their product is
$(x-y)\cdots(x-q^{n-1}y)$)
$$
\begin{matrix}
  (x-y)(x-qy)\cdots (x-q^{p^{r}-1}y), \\
  (x-q^{p^{r}}y)(x-q^{p^{r}+1}y)\cdots (x-q^{2p^{r}-1}y), \\
  \vdots \\
  (x-q^{(a_r-1)p^{r}}y)(x-q^{(a_r-1)p^r+1}y)\cdots
  (x-q^{a_rp^{r}-1}y),
\end{matrix}
$$
is divisible by $\varphi^{r-1}(\tilxi)$.  For this recall the
congruence (\Cref{eq:congruence-for-q-divided-powers})
$$
x^{p^{r}}-y^{p^{r}}\equiv (x-y)(x-qy)\cdots (x-q^{p^{r}-1}y)\ \mathrm{
  mod }\ \varphi^{r-1}(\tilxi).
$$
Replacing in this congruence $y$ by
$q^{p^{r}}y,\ldots, q^{(a_r-1)p^{r}}y$ shows that each of the above
$a_r$ elements is congruent modulo $\varphi^{r-1}(\tilxi)$ to an
element of the form
$$
x^{p^{r}}-q^{k}y^{p^{r}}
$$
with $k\geq 0$ divisible by $p^{r}$. But we have
$$
x^{p^{r}}-q^{k}y^{p^{r}}=(x^{p^{r}}-y^{p^{r}})+y^{p^{r}}(1-q^{k})
$$
and we claim that under our assumptions both summands are divisible by
$\varphi^{r-1}(\tilxi)$. For the first summand we use that $x,y$ are
of rank $1$ to write
$$
x^{p^{r}}-y^{p^{r}}=\varphi^{r-1}(x^p-y^p)=\varphi^{r-1}(\tilxi)\varphi^{r-1}(\frac{x^p-y^p}{\tilxi})
$$
which makes sense as we assumed that
$$
{x^p-y^p}\in \tilxi A.
$$
For the second summand we note that
$$
1-q^k=\frac{1-q^k}{1-q^{p^{r}}}\varphi^{r-1}(\tilxi)(1-q^{p^{r-1}})
$$
with all factors in $\Z_p[[q-1]]$ as $p^r$ divides $k$.  It remains to
prove that
$$
\gamma_{n,q}(x-y)=\frac{(x-y)(x-qy)\cdots(x-q^{n-1}y)}{[n]_q!}
$$
lies in $\mathcal{N}^{\geq n} A$.  But
$$
\varphi(\gamma_{n,q})=\frac{(x^p-y^p)(x^p-q^py^p)\cdots
  (x^p-q^{p(n-1)}y^p)}{\varphi([n]_q!)}
$$
and as we saw above $\tilxi$ divides each of the $n$ factors
$$
(x^p-y^p),\ (x^p-q^py^p),\ \cdots, (x^p-q^{p(n-1)}y^p).
$$
But
$\tilxi$ and $\varphi([n]_q!)$ form a regular sequence by \Cref{lemma_transversal_implies_regular_sequence} which
implies that
$$
(x^p-y^p)(x^p-q^py^p)\cdots (x^p-q^{p(n-1)}y^p)
$$
is divisible by $\tilxi^n\varphi([n]_q)$ as was to be proven.  This
finishes the proof of the proposition.
\end{proof}

  As the proof shows there exists unique choice of a $q$-divided power
  $$
  \gamma_{n,q}(x-y)
  $$
  which is functorial in the triple $(A,x,y)$ (with $x,y\in A$ satisfying the assumptions in \Cref{sec:q-logarithm-proposition-q-divided-powers-for-rank-1-elements}). From now on we will always assume that these $q$-divided powers are chosen.

Moreover, we get the following lemma concerning the convergence of the
$q$-logarithm.

\begin{lemma}
  \label{sec:q-logarithm-convergence-of-q-logarithm}
  Let $(A,I)$ be a prism over $(\Z_p[[q-1]],(\tilxi))$. Then for every
  element $x\in 1+\mathcal{N}^{\geq 1} A$ of rank $1$ the series
$$
\mathrm{log}_q(x)=\sum\limits_{n=1}^\infty
(-1)^{n-1}q^{-n(n-1)/2} [n-1]_q!\gamma_{n,q}(x-1)
$$
is well-defined and converges in $A$. Moreover,
$\log_q(x)\in \mathcal{N}^{\geq 1}A$ and
$$
\log_q(x)\equiv x-1 \ \mathrm{ mod }\ \mathcal{N}^{\geq 2}A
$$
and
  $$
\log_q(xy)=\log_q(x)+\log_q(y)
$$
for any $x,y\in 1+\mathcal{N}^{\geq 1} A$ of rank $1$.
\end{lemma}
\begin{proof}
  By our assumption on $x$ we get $\varphi(x-1)\in \tilxi A$ and thus
  we may apply
  \Cref{sec:q-logarithm-proposition-q-divided-powers-for-rank-1-elements}
  to $x=x$ and $y=1$. Thus the (canonical choice of) $q$-divided powers
$$
\gamma_{n,q}(x-1)=\frac{(x-1)(x-q)\cdots (x-q^{n-1})}{[n]_q!}
$$
in $A$ are well-defined. Moreover, as
$$
\log_q(x)=\sum\limits_{n=1}^\infty
(-1)^{n-1}q^{-n(n-1)/2}[n-1]_q!\gamma_{n,q}(x-1)
$$
and the elements $[n-1]_q!$ tend to zero in $A$ for the $(p,I)$-adic topology we can conclude that
the series $\mathrm{log}_q(x)$ converges because $A$ is
$\tilxi$-adically complete. The claim concerning the Nygaard
filtrations follows directly from
$\gamma_{n,q}(x-1)\in \mathcal{N}^{\geq n} A$, which was proven in
\Cref{sec:q-logarithm-proposition-q-divided-powers-for-rank-1-elements}.
That $\log_q$ is a homomorphism can be seen in the universal case in which $A$ is flat over $\Z_p[[q-1]]$ (by \cite[Proposition 3.13]{bhatt_scholze_prisms_and_prismatic_cohomology}). Then the formula $\log_q(xy)=\log_q(x)+\log_q(y)$ can be checked after base change to $\Q_p[[q-1]]$ where it follows from \Cref{lemma_properties_q_logarithm} as the usual logarithm is a homomorphism.
\end{proof}

\section{Prismatic cohomology and topological cyclic
  homology}\label{sec:prism-cohom-topol}

This section is devoted to the relation of the prismatic
cohomology developed by Bhatt and Scholze \cite{bhatt_scholze_prisms_and_prismatic_cohomology} with
topological cyclic homology (as described by Bhatt, Morrow and Scholze \cite{bhatt_morrow_scholze_topological_hochschild_homology}) following \cite[Section 11.5.]{bhatt_scholze_prisms_and_prismatic_cohomology}.

Let $R$ be a quasi-regular semiperfectoid ring (cf.\ \cite[Definition 4.19.]{bhatt_morrow_scholze_topological_hochschild_homology}), and let $S$ be any perfectoid ring with a map $S \to R$. 
\begin{proposition}
  \label{sec:prism-cohom-topol-1-initial-object-for-quasi-regular-semiperfectoids}
The category of prisms $(A,I)$ with a map $R \to A/I$ admits an initial object $(\prism_R^{\rm init},I)$, which is a bounded prism. Moreover, $\prism_R^{\rm init}$ identifies with the derived prismatic  cohomology $\prism_{R/A_{\rm inf}(S)}$, for any choice of $S$ as before. 
\end{proposition}
\begin{proof}
See \cite[Proposition 7.2, Proposition 7.10]{bhatt_scholze_prisms_and_prismatic_cohomology} or \cite[Proposition 3.4.2]{anschuetz_le_bras_prismatic_dieudonne_theory}.
\end{proof}

In the following, we simply write $\prism_R=\prism_R^{\rm init}=\prism_{R/A_{\rm inf}(S)}$.

\begin{theorem}
  \label{sec:prism-cohom-topol-1-comparsion-initial-and-topological-prism}
  Let $R$ be a quasi-regular semiperfectoid ring. There is a functorial (in $R$) $\delta$-ring structure on $\widehat{\prism}_R^{\mathrm{top}}:=\pi_0(\mathrm{TC}^{-}(R;\Z_p))$ refining the cyclotomic Frobenius. The induced map $\prism_R=\prism_R^{\rm init} \to \widehat{\prism}_R^{\mathrm{top}}$ identifies $\widehat{\prism}_R^{\mathrm{top}}$ with the completion with respect to the Nygaard filtration (\Cref{sec:transversal-prisms-definition-nygaard-filatration}) of $\prism_R$, and is compatible with the Nygaard filtration on both sides.
\end{theorem}
\begin{proof}
 See \cite[Theorem 11.10]{bhatt_scholze_prisms_and_prismatic_cohomology}.
\end{proof}

The Nygaard filtration on $\widehat{\prism}_R^{\mathrm{top}}$ is defined as the double-speed abutment filtration for the (degenerating) homotopy fixed point spectral sequence
$$
E^{ij}_2:=H^i(\mathbb{T},\pi_{-j}(\THH(R;\Z_p)))\Rightarrow \pi_{-i-j}(\TC^{-}(R;\Z_p))
$$
for the $\mathbb{T}=S^1$-action on $\THH(R;\Z_p)$.
If $\eta \in H^2(\mathbb{T},\Z)$ is a generator, then multiplication by any lift $v\in \pi_{-2}(\TC^{-}(R;\Z_p))$ of the image of $\eta$ in $H^2(\mathbb{T},\pi_0(\THH(R;\Z_p)))$ induces isomorphisms
$$
\pi_{2i}(\TC^{-}(R;\Z_p)))\cong \mathcal{N}^{\geq i}\widehat{\prism}_R^{\mathrm{top}}
$$
for $i\in \Z$.

\begin{remark}
  \label{remark-delta-structure-on-nygaard-completed-prism}

We will only use the fact that $\widehat{\prism}_R$ is a prism in this paper (as we will apply the results of \Cref{sec:some-results-prisms} to $\pi_0(\TC^{-}(R;\Z_p))$) and that the topological Nygaard filtration, defined via the homotopy fixed point spectral sequence, agrees with the Nygaard filtration from \Cref{sec:transversal-prisms-definition-nygaard-filatration}, but the way one proves this is by showing the stronger statement that $\widehat{\prism}_R^{\mathrm{top}}$ is the Nygaard completion of $\prism_R$. We ignore if there is a more direct way to produce the $\delta$-structure on $\widehat{\prism}_R$ (cf.\ \cite[Remark 1.14.]{bhatt_scholze_prisms_and_prismatic_cohomology}). 
\end{remark}

\section{The $p$-completed cyclotomic trace in degree $2$}
\label{sec:p-compl-cycl}

Now we are settled to prove our main theorem on the identification of
the $p$-completed cyclotomic trace. Recall that for any ring $A$ the cyclotomic trace
$$
\mathrm{ctr}\colon K(A)\to \TC(A)
$$
from the algebraic $K$-theory of $A$ to its topological cyclic homology is a natural morphism\footnote{When upgraded to a natural transformation of functors on small stable $\infty$-categories the cyclotomic trace is uniquely determined by these properties, cf.\ \cite[Section 10.3.]{blumberg_gepner_tabuada_a_universal_characterization_of_higher_algebraic_k_theory}.} refining the Dennis trace $\mathrm{Dtr}\colon K(A)\to \HH(A)$ introduced in \Cref{sec:p-completed-dennis}, cf.\  \cite[Section 10.3]{blumberg_gepner_tabuada_a_universal_characterization_of_higher_algebraic_k_theory}, \cite[Section 5]{boekstedt_hsiang_madsen_the_cyclotomic_trace_and_algebraic_k_theory_of_spaces}.
 Let us carefully fix some notation.
For the whole section we fix a generator $\gamma\in H_1(\mathbb{T},\Z)$, but note that the formulas in \Cref{sec:conclusion-theorem-identification-of-cyclotomic-trace} will be independent of this choice.
  Set $\Z_p^{\mathrm{cycl}}$ as the $p$-completion of
$\Z_p[\mu_{p^\infty}]$ and choose some $p$-power compatible system of
$p$-power roots of unity
$$
\varepsilon:=(1,\zeta_p,\zeta_{p^2},\ldots)\in
(\Z_p^{\mathrm{cycl}})^\flat
$$
  with $\zeta_p\neq 1$. This choice determines several elements as we will now discuss.
Set
$$
q:=[\varepsilon]_{\theta}\in A_\inf(\Z_p^\cycl):=W((\Z_p^{\mathrm{cycl}})^\flat)\cong \pi_0(\TC^-(\Z_p^\cycl;\Z_p)),
$$
$$
\mu:=q-1,
$$
$$
\tilxi:=[p]_q=\frac{q^p-1}{q-1}=1+q+\ldots+q^{p-1}.
$$
and
$$
\xi=\varphi^{-1}(\tilxi).
$$
\blue{Note that the ring $A_\inf(\Z_p^\cycl)$ is the $(p,q-1)$-adic completion of $\Z_p[q^{1/p^{\infty}}]$.}
We now construct elements
$$
u\in \pi_{2}(\TC^-(\Z_p^{\cycl};\Z_p)),
$$
$$
v\in \pi_{-2}(\TC^-(\Z_p^\cycl;\Z_p))
$$
such that $uv=\xi\in \pi_0(\TC^-(\Z_p^{\cycl};\Z_p))$\footnote{We need a finer statement than \cite[Proposition 6.2 and Proposition 6.3]{bhatt_morrow_scholze_topological_hochschild_homology} which asserts the existence of some $u,v$ as above with $uv=a\xi$ for some unspecified unit $a\in A_{\inf}(\Z_p^\cycl)^\times$.}. The elements $u,v$ will be uniquely determined by $\varepsilon$.
Let
$$
\mathrm{ctr}\colon T_p(\Z_p^\cycl)^\times \to \pi_2(\TC(\Z_p^\cycl;\Z_p))
$$
be the cyclotomic trace in degree $2$. We denote by the same symbol the composition
$$
\mathrm{ctr}\colon T_p(\Z_p^\cycl)^\times \to \pi_2(\TC^{-}(\Z_p^\cycl;\Z_p))
$$
with the canonical morphism $\TC(-;\Z_p)\to \TC^-(-;\Z_p)$.
Let
$$
\mathrm{can}\colon \TC^{-}(-;\Z_p)\to \TP(-;\Z_p)
$$
be the canonical morphism (from homotopy to Tate fixed points).

\begin{lemma}
  \label{sec:p-compl-cycl-1-lemma-cyclotomic-trace-of-epsilon-divisible-by-mu}
  The element
  $$
  \mathrm{can}(\mathrm{ctr}(\varepsilon))\in \pi_2(\TP(\Z_p^\cycl;\Z_p))
  $$
  is divisible by $\mu$.
\end{lemma}

A similar statement (in terms of $\mathrm{TF}$) is proven in \cite[Proposition 2.4.2]{hesselholt_on_the_topological_cyclic_homology_of_the_algebraic_closure_of_a_local_field} (cf.\ \cite[Definition 4.1]{hesselholt_topological_hochschild_homology_and_the_hasse_weil_zeta_function}) using the explicit description of the cyclotomic trace in degree $1$ via $\mathrm{TR}$ from \cite[Lemma 4.2.3.]{geisser_hesselholt_topological_cyclic_homology_of_schemes}.

\begin{proof}
  Fix a generator
  $$
  \sigma^\prime\in \pi_2(\TP(\Z_p^\cycl;\Z_p)).
  $$
  It suffices to show that $\mathrm{can}\circ\mathrm{ctr}(\varepsilon)$ maps to $0$ under the composition
  $$
  \pi_2(\mathrm{TP}(\Z_p^\cycl;\Z_p))\xrightarrow{{\sigma^\prime}^{-1}}\pi_0(\TP(\Z_p^\cycl;\Z_p))\cong A_\inf(\Z_p^\cycl)\to W(\Z_p^\cycl)
  $$
  because the kernel of $A_\inf(\Z_p^\cycl)\to W(\Z_p^\cycl)$ is generated by $\mu$ (cf.\ \cite[Lemma 3.23]{bhatt_morrow_scholze_integral_p_adic_hodge_theory}).
  It therefore suffices to prove the statement for $\mathcal{O}_C$ for $C/\Q_p^\cycl$ an algebraically closed, complete non-archimedean extension. Over $\mathcal{O}_C$ we can (after changing $\sigma^\prime$) find
  $$
  u^\prime\in \pi_2(\TC^{-}(\mathcal{O}_C;\Z_p)), 
  $$
  $$
  v^\prime\in \pi_{-2}(\TC^{-}(\mathcal{O}_C;\Z_p))
  $$
  such that
  $$
  u^\prime v^\prime=\xi=\frac{\mu}{\varphi^{-1}(\mu)},
  $$
  $$
  \mathrm{can}(v^\prime)={\sigma^{\prime}}^{-1}
  $$
  and the cyclotomic Frobenius maps $u^\prime$ to $\sigma^\prime$, cf.\ \cite[Proposition 6.2., Proposition 6.3.]{bhatt_morrow_scholze_topological_hochschild_homology}. Then multiplication by $v^\prime$ induces an isomorphism
  $$
  \pi_2(\TC(\mathcal{O}_C;\Z_p))\cong A_{\inf}(\mathcal{O}_C)^{\varphi=\tilxi}.
  $$
  By \cite[Proposition 6.2.10.]{fargues_fontaine_courbes_et_fibres_vectoriels_en_theorie_de_hodge_p_adique}
  $$
  (A_\inf(\mathcal{O}_C)[1/p])^{\varphi=\tilxi}
  $$
  is $1$-dimensional over $\Q_p$ and thus generated by $\mu$ (as $\mu\neq 0$ and $\varphi(\mu)=\tilxi\mu$). But $\mu$ is not divisible by $p$ in $A_\inf(\mathcal{O}_C)$ as it maps to a unit in $W(C)$. This proves that $A_\inf(\mathcal{O}_C)^{\varphi=\tilxi}=\Z_p\mu$, which implies the claim.
\end{proof}

Let us define
$$
\sigma:=\frac{\mathrm{ctr}(\varepsilon)}{\mu}\in \pi_2(\TP(\Z_p^\cycl;\Z_p))
$$
and
$$
u:=\xi\sigma\in \pi_2(\TC^{-}(\Z_p^\cycl)).
$$
More precisely, the element $u$ is defined via $\mathrm{can}(u)=\xi\sigma$ (note that $\xi\sigma$ lies indeed in the image of
$$
\mathrm{can}\colon \pi_2(\TC^{-}(\Z_p^\cycl;\Z_p))\to \pi_2(\TP(\Z_p^\cycl;\Z_p))
$$
as the abutment filtration for the Tate fixed point spectral sequence on $\pi_2(\TP(\Z_p^\cycl;\Z_p))$ is the $\xi$-adic filtration).

\begin{lemma}
  \label{sec:p-compl-cycl-2-chosen-u-lifts-xi}
  The element $u$ defined above lifts the class of
  $$
  \delta_1(\gamma)\xi\in \pi_2(\THH(\Z_p^\cycl;\Z_p))\cong \pi_2(\HH(\Z_p^\cycl;\Z_p)) \overset{\alpha_\gamma}{\cong} (\xi)/(\xi^2).
  $$
\end{lemma}
\begin{proof}
  By definition
  $$
  \mathrm{can}(u)=\frac{\xi}{\mu}{\mathrm{ctr}(\varepsilon)}\in \pi_2(\mathrm{TP}(\Z_p^\cycl;\Z_p)).
  $$
  Now
  $$
  \frac{\xi}{\mu}=\frac{1}{\varphi^{-1}(\mu)}
  $$
  and $(\xi)/(\xi^2)$ is $\varphi^{-1}(\mu)$-torsion free as a module over $A_\inf(\Z_p^\cycl)$ (because
  $$
  \theta(\varphi^{-1}(\mu))=\zeta_p-1\neq 0\in \Z_p^\cycl).
  $$

  Moreover, the cyclotomic trace lifts the Dennis trace in Hochschild homology.
  Thus by \Cref{proposition_second_description_of_dennis_trace},
  $$
  \alpha_\gamma(\mathrm{Dtr}(\varepsilon))\equiv \delta_1(\gamma)([\varepsilon]-1) \in (\xi)/(\xi^2)
  $$
 and therefore
  $$
  u\equiv \delta_1(\gamma)\frac{[\varepsilon]-1}{\varphi^{-1}(\mu)}=\delta_1(\gamma)\xi \in (\xi)/(\xi^2)
  $$
  as desired.
\end{proof}

In particular, we see that the element
$$
\sigma\in \pi_2(\TP(\Z_p^\cycl;\Z_p))
$$
is a generator.
Set
$$
v:=\sigma^{-1}\in \pi_{-2}(\TC^-(\Z_p^\cycl;\Z_p))\overset{\mathrm{can}}{\cong}\pi_2(\TP(\Z_p^\cycl;\Z_p)).
$$
Then
$$
uv=\xi.
$$

Recall that for any morphism of rings $R\to A$ the negative cyclic homology is defined to be
$$
\HC^{-}(A/R):=\HH(A/R)^{h\mathbb{T}}
$$
where $(-)^{h\mathbb{T}}:=\varprojlim\limits_{B\mathbb{T}}(-)$, cf.\ \cite{hoyois_the_homotopy_fixed_points_of_the_circle_action_on_hochschild_homology} for a comparison with the classical definition in \cite[Definition 5.1.3]{loday_cyclic_homology}.
The homotopy fixed point spectral sequence
$$
H^i(B\mathbb{T},\pi_{-j}(\HH(A/R))\Rightarrow \pi_{-i-j}(\HC^{-}(A/R))
$$
endows $\pi_{\ast}(\HC^{-}(A/R))$ with a (multiplicative) decreasing filtration which we denote by
$$
\mathcal{N}^{\geq \bullet}\HC^{-}(A/R).
$$
We note that each generator $\gamma\in H_1(\mathbb{T},\Z)$ defines canonically a generator $\eta_\gamma\in H^2(B\mathbb{T},\Z)$. We will do the abuse of notation and denote by $\gamma\in H_1(\mathbb{T},R)$ the image of $\gamma\in H_1(\mathbb{T},\Z)$; similarly, for $\eta_\gamma$.

\begin{proposition}
  \label{sec:p-completed-dennis-1-proposition-making-hkr-and-so-on-explicit}
  Let $\gamma\in H_1(\mathbb{T},\Z)$ be a generator and assume $A=R/(f)$ for some non-zero divisor $f\in R$.  Then
  \begin{enumerate}
  \item $\HH_\ast(A/R)$ is concentrated in even degrees and the homotopy fixed point spectral sequence
    $$
    H^i(B\mathbb{T},\pi_{-j}(\HH(A/R))\Rightarrow \pi_{-i-j}(\HC^{-}(A/R))
    $$
    degenerates.
  \item There exists a unique element $\delta_2\in \{\pm 1\}$, independent of the choice of $\gamma$, such that the morphism
    $$
    (f)/(f)^2\xrightarrow{\alpha_\gamma} \pi_2(\HH(A/R))\xrightarrow{\tilde{\eta}_\gamma} \pi_0(\HC^{-}(A/R))/\mathcal{N}^{\geq 2}\HC^{-}(A/R)
    $$
    sends the class of $f$ to $\delta_2f\cdot 1_{\pi_0(\HC^{-}(A/R))/\mathcal{N}^{\geq 2}\HC^{-}(A/R)}$. Here the first isomorphism is the one of \Cref{lemma_second_hochschild_homology_for_a_surjection}. The second morphism is the multiplication by some lift $\tilde{\eta}_\gamma\in \pi_{-2}(\HC^{-}(A/R))$ of $\eta_\gamma\in H^2(B\mathbb{T},\pi_0(\HH(A/R))$\footnote{As we mod out by $\mathcal{N}^{\geq 2}$ and the spectral sequence degenerates, the second morphism does not depend on the choice of a lift.}.
    
  \end{enumerate}
\end{proposition}
\begin{proof}
  The first claim follows from the HKR-filtration as the exterior powers
  $$
  \wedge^iL_{A/R}[i]
  $$
  are concentrated in even degrees for all $i\geq 0$.
  For the second claim we can reduce by naturality to the universal case $R=\Z[x], f=x$ in which case it is well-known that the elements
  $$
  \tilde{\eta}_\gamma(\alpha_\gamma(f)),\ f\cdot 1_{\pi_0(\HC^{-}(A/R))/\mathcal{N}^{\geq 2}\HC^{-}(A/R)}
  $$
  are generators of the free $A$-module $\mathcal{N}^{\geq 1}\pi_0(\HC^-(A/R)/\mathcal{N}^{\geq 2}\pi_0(\HC^-(A/R)))$ of rank $1$. This implies the existence of $\delta_2$ as $A\cong \Z$. As the composition $\tilde{\eta}_\gamma\circ \alpha_\gamma$ is independent of the choice of $\gamma\in H_1(\mathbb{T},\Z)$ (because both $\alpha_\gamma$ and $\tilde{\eta}_\gamma$ will be changed by a sign), the proof is finished.
\end{proof}

\begin{remark}
  \label{sec:p-completed-dennis-1-remark-on-delta-1-for-spectral-sequence}
  We expect that $\delta_2=1$, but did not make the explicit computation, since we will not need it. 
\end{remark}

We need the following relation of $v$ to $\eta_\gamma$.
\begin{lemma}
  \label{sec:p-compl-cycl-1-images-of-v-and-eta-gamma}
Let $\underline{\eta_\gamma}, \underline{v}\in H^2(\mathbb{T},\pi_0(\HH(\Z_p^{\cycl};\Z_p)))$ be the images of $\eta_\gamma\in H^2(\mathbb{T},\Z)$ resp.\ $v\in \pi_{-2}(\TC^{-}(\Z_p^\cycl;\Z_p))$ under the canonical morphisms 
$$
H^2(\mathbb{T},\Z)\to H^2(\mathbb{T},\pi_0(\HH(\Z_p^{\cycl};\Z_p)))
$$ 
resp.\ 
$$
\pi_{-2}(\TC^{-}(\Z_p^\cycl;\Z_p))\to H^2(\mathbb{T},\pi_0(\HH(\Z_p^{\cycl};\Z_p)))
.$$ 
Then
$$
\underline{v}=\delta_2 \delta_{1}(\gamma)\underline{\eta_\gamma}.
$$
\end{lemma}
\begin{proof}
  By \Cref{sec:p-compl-cycl-2-chosen-u-lifts-xi} we know that the image of $u$ in 
  $$
  \pi_2(\HH(\Z_p^\cycl;\Z_p)) \overset{\alpha_\gamma}{\cong} (\xi)/(\xi^2)
  $$
  is
$$
\alpha_\gamma(\delta_1(\gamma)\xi).
$$
As $\underline{\eta_\gamma}, \underline{v}\in H^2(\mathbb{T},\pi_0(\HH(\Z_p^{\cycl};\Z_p)))$ there exists some unit $r \in \Z_p^\cycl$ such that $r \underline{\eta_\gamma}=\underline{v}$. We can calculate in $\pi_0(\HC^{-}(\Z_p^\cycl;\Z_p))/\mathcal{N}^{\geq 2}\pi_0(\HC^{-}(\Z_p^\cycl;\Z_p))$
$$
\begin{matrix}
  \xi  =  uv  =  v\alpha_{\gamma}(\delta_1(\gamma)\xi)\\
 =  r\eta_\gamma \alpha_{\gamma}(\delta_1(\gamma)\xi) \\
=  r\delta_2\delta_1(\gamma) \xi
\end{matrix}
$$ 
using \Cref{sec:p-completed-dennis-1-proposition-making-hkr-and-so-on-explicit}.
Thus, $r=\delta_2\delta_1(\gamma)$.
\end{proof}

One has the following (important) additional property (which, up to changing $\xi$ by some unit, is implied by the conjunction of \cite[Proposition 6.2., Proposition 6.3.]{bhatt_morrow_scholze_topological_hochschild_homology}).

\begin{lemma}
  \label{sec:p-compl-cycl-3-varphi-sends-u-to-sigma}
  The cyclotomic Frobenius
  $$
  \varphi^{h\mathbb{T}}\colon \pi_2(\TC^-(\Z_p^{\cycl};\Z_p))\to \pi_2(\TP(\Z_p^\cycl;\Z_p))
  $$
  sends $u$ to $\sigma$.
\end{lemma}
\begin{proof}
  The cyclotomic Frobenius $\varphi^{h\mathbb{T}}$ is linear over the Frobenius on $A_\inf$. Thus we can calculate (note $\frac{\xi}{\mu}=\frac{1}{\varphi^{-1}(\mu)}$)
  $$
  \varphi^{h\mathbb{T}}(u)=\varphi(\frac{\xi}{\mu})\varphi^{h\mathbb{T}}(\mathrm{ctr}(\varepsilon))=\frac{1}{\mu}\varphi^{h\mathbb{T}}(\mathrm{ctr}(\varepsilon)).
  $$
  But
  $$
  \varphi^{h\mathbb{T}}(\mathrm{ctr}(\varepsilon))=\mathrm{can}(\mathrm{ctr}(\varepsilon))
  $$
  as the cyclotomic trace has image in $\pi_2(\TC(\Z_p^\cycl;\Z_p))$. This implies that
  $$
  \varphi^{h\mathbb{T}}(u)=\frac{\mathrm{ctr}(\varepsilon)}{\mu}=\sigma
  $$
  as desired.
\end{proof}

By \Cref{sec:p-compl-cycl-3-varphi-sends-u-to-sigma} one can conclude that there is a commutative diagram, whose vertical arrows are isomorphisms,
$$
\xymatrix{
  \pi_2(\TC(R;\Z_p))\ar[r]\ar[d]^{\beta_\varepsilon} & \pi_2(\TC^-(R;\Z_p)) \ar[r]^{\varphi^{h\mathbb{T}}-\mathrm{can}}\ar[d]^{v} & \pi_2(\TP(R;\Z_p))\ar[d]^{\sigma^{-1}}\\
  \widehat{\prism}_R^{\varphi=\tilxi}\ar[r] & \mathcal{N}^{\geq 1}\widehat{\prism}_R \ar[r]^{\frac{\varphi}{\tilxi}-1} & \widehat{\prism}_R
}
$$
for any quasi-regular semiperfectoid $\Z_p^\cycl$-algebra $R$. We remind the reader that the induced isomorphism
$$
\beta_\varepsilon \colon \pi_2(\TC(R;\Z_p))\cong \widehat{\prism}_R^{\varphi=\tilxi}
$$
depends only on $\varepsilon$, not on $\gamma$.

For a quasi-regular semiperfectoid ring $R$ we denote by
$$
[-]_{\tilde{\theta}}\colon R^\flat=\varprojlim\limits_{x\mapsto x^p}R \to \prism_R
$$
the Teichm\"uller lift. More precisely, the canonical morphism $R\to \overline{\prism}_R$ induces a morphism $\iota\colon R^\flat\to \overline{\prism}_R^\flat$ and $[-]_{\tilde{\theta}}$ is the composition of $\iota$ with the Teichm\"uller lift for the surjection
$$
\prism_R\to \overline{\prism}_R.
$$
We set\footnote{This agrees with the definition of $[-]_{\theta}$ made in the introduction.}
$$
[-]_{\theta}:=[(-)^{1/p}]_{\tilde{\theta}}.
$$
We will consider the $p$-adic Tate module
$$
T_pR^\times=\varprojlim\limits_{n\geq 0} R^\times[p^n]
$$
of $R^\times$
as being embedded into $R^\flat$ as the elements with first coordinate
equal to $1$.

We are ready to state and prove our main theorem.

\begin{theorem}
  \label{sec:conclusion-theorem-identification-of-cyclotomic-trace}
  Let $R$ be a quasi-regular semiperfectoid
  $\Z_p^{\mathrm{cycl}}$-algebra. Then the composition
$$
T_pR^\times\to
\pi_2(K(R;\Z_p))\xrightarrow{\mathrm{ctr}}\pi_2(\TC(R;\Z_p))\overset{\beta_\varepsilon}\cong
\widehat{\prism}^{\varphi=\tilxi}_R
$$
is given by sending $x\in T_p(R^\times)$ to
$$
\log_q([x]_{\theta})=\sum\limits_{n=1}^\infty
(-1)^{n-1}q^{-n(n-1)/2}\frac{([x]_{\theta}-1)([x]_{\theta}-q)\cdots
  ([x]_{\theta}-q^{n-1})}{[n]_q}.
$$
\end{theorem}
\begin{proof}
  Replacing $R$ by the universal case $\Z_p^{\mathrm{cycl}}\langle x^{1/p^\infty}\rangle/(x-1)$ we may assume that $R$ is $p$-torsion free and (thus) that $(\widehat{\prism}_R,(\tilxi))$ is transversal (by \Cref{lemma_p_torsionfree_implies_nonzerodivisor} it suffices to see that $(p,\xi)$ is a regular sequence which follows as $\widehat{\prism}_R/\xi\cong \widehat{L\Omega_{R/\Z_p^\cycl}}$, by \cite[Theorem 7.2.(5)]{bhatt_morrow_scholze_topological_hochschild_homology}, is $p$-torsion free).

  Let us define
$$
\mathrm{ctr}_2\colon T_pR^\times\to
\pi_2(K(R;\Z_p))\xrightarrow{\mathrm{ctr}}\pi_2(\TC(R;\Z_p)).
$$
By \Cref{sec:prism-cohom-topol-1-comparsion-initial-and-topological-prism} the canonical morphism
$$
\iota\colon \prism_R\to \pi_0(\TC^{-}(R;\Z_p))
$$
is compatible with the Nygaard filtrations and identifies
$\pi_0(\TC^{-}(R;\Z_p))$ with the Nygaard completion $\widehat{\prism}_R$ of $\prism_R$.
By \Cref{sec:transversal-prisms-corollary-injectivity-for-reduction-modulo-second-nygaard-filtration} the morphism
$$
\prism^{\varphi=\tilxi}_R\hookrightarrow \mathcal{N}^{\geq 1}\prism_R/\mathcal{N}^{\geq 2}\prism_R\cong \mathcal{N}^{\geq 1}\widehat{\prism}_R/\mathcal{N}^{\geq 2}\widehat{\prism}_R
$$
is injective.
Hence it suffices to show that the two morphisms $\mathrm{log}_q([-]_\theta)$ and $\beta_\varepsilon \circ \mathrm{ctr}$ agree modulo $\mathcal{N}^{\geq 2}\widehat{\prism}_R$.
Multiplication by the element $v\in \pi_{-2}(\TC^{-}(\Z_p^{\cycl};\Z_p))$ constructed after \Cref{sec:p-compl-cycl-2-chosen-u-lifts-xi} and the HKR-isomorphism (which depends on $\gamma$) induce an isomorphism
$$
J/J^2 \overset{\alpha_\gamma}{\cong} \pi_2(\THH(R;\Z_p))\overset{v}{\cong}\mathcal{N}^{\geq 1}\widehat{\prism}_R/\mathcal{N}^{\geq 2}\widehat{\prism}_R
$$
where $J$ is the kernel of the surjection
$$
\theta\colon W(R^\flat)\to R.
$$
By \Cref{sec:p-completed-dennis-1-proposition-making-hkr-and-so-on-explicit} and  \Cref{sec:p-compl-cycl-1-images-of-v-and-eta-gamma}
this isomorphism sends the class of $j\in J$ to 
$$
\delta_2^2\delta_1(\gamma)\cdot j\cdot 1_{\widehat{\prism}_R/\mathcal{N}^{\geq 2}\widehat{\prism}_R}=\delta_1(\gamma)\cdot j\cdot 1_{\widehat{\prism}_R/\mathcal{N}^{\geq 2}\widehat{\prism}_R}.
$$ 
for the canonical $W(R^{\flat})$-algebra structure on
$$
\widehat{\prism}_R/\mathcal{N}^{\geq 2}\widehat{\prism}_R\cong \pi_0(\TC^{-}(R;\Z_p))/\mathcal{N}^{\geq 2}\pi_0(\TC^-(R;\Z_p))
$$
$$
\cong \pi_0(\HC^-(R/W(R^\flat)))/\mathcal{N}^{\geq 2}\pi_0(\HC^-(R/W(R^\flat)))
$$
(which lifts the morphism $\theta$).
Let $x\in T_pR^\times$.
By \Cref{sec:q-logarithm-convergence-of-q-logarithm}
$$
\mathrm{log}_q([x]_\theta)\equiv [x]_\theta-1 \textrm{ mod } \mathcal{N}^{\geq 2}\widehat{\prism}_R.
$$
On the other hand, as the cyclotomic trace reduces to the Dennis trace $\mathrm{Dtr}$,
we can calculate using \Cref{proposition_second_description_of_dennis_trace} and \Cref{sec:p-compl-cycl-1-images-of-v-and-eta-gamma}
$$
\beta_{\varepsilon}(\mathrm{ctr}(x))\equiv v\mathrm{Dtr}(x)
$$
$$
=v \delta_1(\gamma)([x]_{\theta}-1) =\delta_1(\gamma)^2([x]_\theta-1)\cdot 1_{\widehat{\prism}_R/\mathcal{N}^{\geq 2}\widehat{\prism}_R} \textrm{ mod } \mathcal{N}^{\geq 2}\widehat{\prism}_R
$$
$$
=([x]_{\theta}-1) \textrm{ mod } \mathcal{N}^{\geq 2}\widehat{\prism}_R
$$
Thus we can conclude
$$
\mathrm{log}_q([x]_\theta)=\beta_\varepsilon \circ\mathrm{ctr}(x)
$$
as desired.
\end{proof}

\begin{corollary}
\label{sec:conclusion-corollary-bijectivity-qlog}
Let $R$ be a quasi-regular semiperfectoid $\Z_p^{\rm cycl}$-algebra. The map
\[ \log_q([-]_\theta) \colon T_p(R^{\times}) \to \widehat{\prism}_R^{\varphi=\tilxi} \]
is a bijection.
\end{corollary}
\begin{proof}
Since both sides satisfy quasi-syntomic descent\footnote{For $T_p(-)^\times$ this follows from $p$-completely faithfully flat descent on $p$-complete rings with bounded $p^\infty$-torsion, cf.\ \cite[Appendix]{anschuetz_le_bras_prismatic_dieudonne_theory}, for $\widehat{\prism}_R^{\varphi=\tilxi}$ this is is proven in \cite{bhatt_morrow_scholze_topological_hochschild_homology}.}, one can assume, as in \cite[Proposition 7.17]{bhatt_morrow_scholze_topological_hochschild_homology}, that $R$ is $w$-local and such that $R^{\times}$ is divisible. In this case, the map
\[ T_p(R^{\times}) \to \pi_2(K(R;\Z_p)) \]
is a bijection. Moreover, \cite[Corollary 6.9]{clausen_mathew_morrow_k_theory_and_topological_cyclic_homology_of_henselian_pairs} shows that 
\[ \mathrm{ctr} : \pi_2(K(R;\Z_p))\to \pi_2(\TC(R;\Z_p)) \]
is also bijective. As by \Cref{sec:conclusion-theorem-identification-of-cyclotomic-trace}, the composite of these two maps is
the map $\log_q([-]_{\tilde{\theta}})$, this proves the corollary.
\end{proof}

\begin{remark} 
As explained at the end of the introduction, one can give a direct and more elementary proof of \Cref{sec:conclusion-corollary-bijectivity-qlog} when $R$ is the quotient of a perfect ring by a finite regular sequence (\cite{scholze_weinstein_moduli_of_p_divisible_groups}) or when $R$ is a $p$-torsion free quotient of a perfectoid ring by a finite regular sequence and $p$ is odd. But we do not know how to prove it directly in general.
\end{remark}

\bibliography{biblio.bib}
\bibliographystyle{plain}

\end{document}